\documentclass[12pt]{article}

\usepackage{amsmath,amsthm,amssymb} 
\usepackage{mathtools}
\usepackage{mathrsfs}
\usepackage{stmaryrd}
\usepackage{enumerate} 
\usepackage{tikz-cd}
\usepackage{bm}
\usepackage{hyperref}

\newtheorem{theorem}{Theorem}[section]
\newtheorem{lemma}[theorem]{Lemma}
\newtheorem{prop}[theorem]{Proposition}
\newtheorem{cor}[theorem]{Corollary}

\theoremstyle{definition}
\newtheorem{definition}[theorem]{Definition}

\newtheorem{remark}[theorem]{Remark}
\newtheorem{assumption}[theorem]{Assumption}

\DeclareMathOperator\AJ{\operatorname{AJ}}
\DeclareMathOperator\Aut{\operatorname{Aut}}
\DeclareMathOperator\bal{\operatorname{bal}}
\DeclareMathOperator\can{\operatorname{can}}
\DeclareMathOperator\CH{\operatorname{CH}}
\DeclareMathOperator\cris{\operatorname{cris}}
\DeclareMathOperator\dR{\operatorname{dR}}
\DeclareMathOperator\Fil{\operatorname{Fil}}
\DeclareMathOperator\Frac{\operatorname{Frac}}

\DeclareMathOperator\Gal{\operatorname{Gal}}
\DeclareMathOperator\GL{\operatorname{GL}}
\DeclareMathOperator\Hom{\operatorname{Hom}}
\DeclareMathOperator\id{\operatorname{id}}
\DeclareMathOperator\Iw{\operatorname{Iw}}

\DeclareMathOperator\ord{\operatorname{ord}}
\DeclareMathOperator\para{\operatorname{par}}

\DeclareMathOperator\Sel{\operatorname{Sel}}

\DeclareMathOperator\Spec{\operatorname{Spec}}
\DeclareMathOperator\Spf{\operatorname{Spf}}
\DeclareMathOperator\Sym{\operatorname{Sym}}

\newcommand{\cl}{\operatorname{cl}}
\newcommand{\Drig}{\operatorname{\mathbf{D}^\dagger_{\rig}}}
\newcommand{\et}{\operatorname{\mathrm{\acute e t}}}

\newcommand{\Hdg}{\operatorname{Hdg}}
\newcommand{\lcm}{\operatorname{\ell \textrm{cm}}}
\newcommand{\phig}{\operatorname{\varphi, \Gamma}}
\newcommand{\pr}{\operatorname{pr}}
\newcommand{\rig}{\operatorname{rig}}

\newcommand{\bff}{\mathbf{f}}
\newcommand{\bfg}{\mathbf{g}}
\newcommand{\bfh}{\mathbf{h}}

\newcommand{\calH}{\mathcal{H}}

\newcommand{\calO}{\mathcal{O}}
\newcommand{\calU}{\mathcal{U}}
\newcommand{\calV}{\mathcal{V}}
\newcommand{\calW}{\mathcal{W}}
\newcommand{\calX}{\mathcal{X}}

\newcommand\bbC{\mathbb{C}}

\newcommand\bbN{\mathbb{N}}
\newcommand\bbQ{\mathbb{Q}}

\newcommand\bbW{\mathbb{W}}
\newcommand\bbZ{\mathbb{Z}}

\newcommand\frakw{\mathfrak{w}}
\newcommand\frakX{\mathfrak{X}}

\newcommand\scrE{\mathscr{E}}
\newcommand\scrL{\mathscr{L}}
\newcommand\scrO{\mathscr{O}}
\newcommand\scrR{\mathscr{R}}

\title{Triple product $p$-adic $L$-functions for finite slope families and a $p$-adic Gross-Zagier formula}

\author{Ting-Han Huang} 

\begin{document}

\maketitle
\begin{abstract}

    In this paper, we generalize the $p$-adic Gross--Zagier formula of Darmon--Rotger on triple product $p$-adic $L$-functions for Hida families to finite slope families.
    Then we demonstrate an application to a special case of the Bloch--Kato conjecture, following the works of Bertolini--Seveso--Venerucci.

\end{abstract}

\section{Introduction}

In \cite{DR}, a $p$-adic Gross--Zagier formula is proved for triple product $p$-adic $L$-functions attached to ordinary families of modular forms.
More precisely, the formula relates
\begin{enumerate}
    \item[-] special values of the Garrett--Rankin triple product $p$-adic $L$-function for Hida families at classical points lying outside the region of interpolation, to
    \item[-] $p$-adic Abel--Jacobi images of \textit{generalized diagonal cycles} in the product of three Kuga--Sato varieties, evaluated at certain differentials.
\end{enumerate}

The first part of this paper aims to generalize the above result to triple product $p$-adic $L$-functions for finite slope families of modular forms.

Let 
$f \in S_k (N_f, \chi_f)$, $g \in S_\ell (N_g, \chi_g)$, $h \in S_m (N_h, \chi_h)$ be a triple of normalized primitive cuspidal eigenforms and set $N := \ell \mathrm{cm}(N_f, N_g, N_h)$.
We will assume throughout this paper that $N$ is square-free and $\chi_f \cdot \chi_g \cdot \chi_h =1$.

For such a triple $(f, g, h)$, one has the (complex) Garrett--Rankin $L$-function $L(f, g, h; s)$.
It can be viewed as the $L$-function attached to the tensor product 
\[
V(f, g, h) = V(f) \otimes V(g) \otimes V(h)
\]
of the (compatible systems of) $p$-adic Galois representations associated with $f, g$ and $h$ respectively.
This $L$-function admits a functional equation relating its values at $s$ and $k + \ell +m -2-s$.
As a consequence, the order of vanishing of $L(f, g, h; s)$ at the central point $c :=\frac{k + \ell +m -2}{2}$ is determined by the root number $\varepsilon \in \{ \pm 1 \}$ of the functional equation.
The root number $\varepsilon$ can be further written as a product $\prod_{ v \mid N \infty} \varepsilon_v$ of local root numbers $\varepsilon_v \in \{ \pm 1 \}$.
We will make the following assumption throughout this article as in \cite{DR}. 

\noindent \textbf{Assumption} H: The local root numbers $\varepsilon_v =+1$ for all finite primes $v \mid N$.

As indicated in \cite[\S~1]{DR}, this condition is satisfied in a broad collection of settings.
Under Assumption H, $\varepsilon = \varepsilon_\infty$ depends only on the triple of weights $(k, \ell, m)$:
$$\varepsilon_\infty = \begin{cases}
    -1 \quad \textrm{if } (k, \ell, m) \textrm{ is balanced;}\\
    +1 \quad \textrm{if } (k, \ell, m) \textrm{ is unbalanced.}
\end{cases}$$
A triple $(x, y, z) \in \bbN^3$ is called \textit{balanced} if
$2x^* <x +y+ z$ for $x^* \in \{ x, y, z \}$.
A triple that is not balanced is called \textit{unbalanced}, and the largest number is called the dominant weight.

Consequently, the $L$-function necessarily vanishes at the central point $c$ when $(k, \ell, m)$ is balanced.
In this situation, instead of studying the values of $L(f, g, h; s)$, one usually studies the values of its derivatives or its associated $p$-adic $L$-functions and expects them to encode arithmetic information.

Choose a prime $p$ with $p \nmid N$ such that we have finite slope ($p$-adic) families $\mathbf{f}$, $\mathbf{g}$, $\mathbf{h}$ passing through $f, g, h$, defined over certain weight spaces $\Omega_f, \Omega_g, \Omega_h$, respectively.
In the language of rigid analytic geometry, the space $\Omega_\bullet$ for $\bullet \in \{ f, g, h\}$ is a finite rigid analytic cover over an admissible disk of the weight space $\calW := \Hom_{\mathrm{cts}}(\bbZ_p^\times, \bbC_p^\times)$ which contains the integers $\bbZ$ via the identification $k \mapsto (\gamma \mapsto \gamma^k)$.

When $\bff$ is of slope $a \in \bbQ_{\geq 0}$,
a point $x \in \Omega_f$ is said to be classical if its image in $\calW$, denoted by $\kappa(x)$, belongs to $\bbZ_{> a+1}$.
We will often identify a classical weight $x$ with $\kappa(x) \in \bbZ$ by an abuse of notations.
We denote the set of classical points in $\Omega_f$ by $\Omega_{f, \mathrm{cl}}$.
For almost all $x \in \Omega_{f, \mathrm{cl}}$, the specialization of $\mathbf{f}$ at $x \in \Omega_{f, \mathrm{cl}}$, denoted by $\mathbf{f}_x$ (or simply $f_x$), is a normalized eigenform of weight $\kappa(x)$ on $\Gamma_1(N, p) := \Gamma_1(N) \cap \Gamma_0(p)$.
For all but finitely many such $x$, $\bff_x$ is the (finite slope) $p$-stabilization of a normalized eigenform of the same weight on $\Gamma_1(N)$, denoted by $f_x^0$.

When $\mathbf{f}$, $\mathbf{g}$, $\mathbf{h}$ are Hida families, one can view them as maps
$$\mathbf{f}: \Omega_f \rightarrow \bbC_p \llbracket q \rrbracket, \  \mathbf{g}: \Omega_g \rightarrow \bbC_p \llbracket q \rrbracket, \ \mathbf{h}: \Omega_h \rightarrow \bbC_p \llbracket q \rrbracket.$$
The triple product $p$-adic $L$-function for Hida families in \cite{DR} is defined as the ratio
$$\scrL^f_p(\bff, \bfg, \bfh): = \frac{(\bff^*, e_{\ord}(\theta^\bullet \bfg^{[p]} \times \bfh))}{(\bff^*, \bff^*)} ,$$
where $\bff^*$ is image of $\bff$ under the Atkin--Lehner operator, $e_{\ord}$ is the Hida's ordinary projector, $\theta = q \frac{d}{dq}$ is Serre's operator, $\bfg^{[p]}:= (1-VU) \bfg$ is the $p$-depletion of $\bfg$, and $( \phantom{e}, \phantom{e} )$ is the Petersson product.

For finite slope families, one needs a different approach.
A more geometric construction of triple product $p$-adic $L$-functions for finite slope families is developed by F. Andreatta and A. Iovita in \cite{tripleL}. 
Section \ref{section: L-function} is dedicated to recall their work.

Philosophically, one wants to define the triple product $p$-adic $L$-function as the ratio
$$\mathscr{L}^f_p (\mathbf{f}, \mathbf{g}, \mathbf{h}) := \frac{( \mathbf{f}^*, \nabla^\bullet \mathbf{g}^{[p]} \times \mathbf{h} ) }{(\mathbf{f}^*,\mathbf{f}^*)}.$$
For this expression to make sense, one needs to $p$-adically iterate the Gauss--Manin connection $\nabla$.
When $\mathbf{f}$, $\mathbf{g}, \mathbf{h}$ are Hida families, as we saw in the above definition, one may replace $\nabla$ by $\theta = q \frac{d}{dq}$ and work only on the $q$-expansions.
Then there is a straightforward way to iterate $\theta$ to $p$-adic powers.
The case for $\nabla$ is more complicated and requires a careful study on the convergence.
However, we would like to refer the detailed computations to \cite{tripleL} and \cite{twisted_tripleL_finite_slope}.


The triple product $p$-adic $L$-function $\mathscr{L}^f_p (\mathbf{f}, \mathbf{g}, \mathbf{h})$ can be viewed as a three-variable function
$$\mathscr{L}^f_p (\mathbf{f}, \mathbf{g}, \mathbf{h}): \Omega_f \times \Omega_g \times \Omega_h \rightarrow \bbC_p,$$
whose value at $(x, y, z) \in \Omega_f \times \Omega_g \times \Omega_h$ is 
$$\frac{( \mathbf{f}_x^*, \nabla^{-t} \mathbf{g}_y^{[p]} \times \mathbf{h}_z ) }{(\mathbf{f}_x^*,\mathbf{f}_x^*)},$$
where $t \in \Omega$ satisfies $\kappa(x) = \kappa(y) + \kappa(z) -2t$.
As one can observe, the form $\bff$ plays a distinct role in $\mathscr{L}^f_p$. 
What is less obvious is that $\mathscr{L}^f_p (\mathbf{f}, \mathbf{g}, \mathbf{h})$ is symmetric, up to a $\pm$ sign, in the last two components.
Similarly, one can define $\mathscr{L}^g_p = \mathscr{L}^g_p (\bff, \bfg, \bfh)$ and $\mathscr{L}^h_p = \mathscr{L}^h_p (\bff, \bfg, \bfh)$.

Set $\Sigma = \Omega_f \times \Omega_g \times \Omega_h$, $\Sigma_{\cl} = \Omega_{f, \cl} \times \Omega_{g, \cl} \times \Omega_{h, \cl}$ and 
$\Sigma_{\cl}' := \{ (x, y, z) \in \Sigma_{\cl} \mid  \kappa(x) + \kappa(y) + \kappa(z) \equiv 0 \bmod 2 \}$.
The set $\Sigma_{\cl}'$ can be naturally divided into four disjoint subsets:
\begin{align*}
    \Sigma_{\bal} &:= \{ (x, y, z) \in \Sigma_{\cl}' \mid (\kappa(x), \kappa(y), \kappa(z)) \textrm{ is balanced} \}; \\
    \Sigma_{f} &:= \{ (x, y, z) \in \Sigma_{\cl}' \mid \kappa(x) \geq \kappa(y) + \kappa(z) \}; \\
    \Sigma_{g} &:= \{ (x, y, z) \in \Sigma_{\cl}' \mid \kappa(y) \geq \kappa(x) + \kappa(z) \}; \\
    \Sigma_{h} &:= \{ (x, y, z) \in \Sigma_{\cl}' \mid \kappa(z) \geq \kappa(x) + \kappa(y) \}.
\end{align*}
The $p$-adic $L$-function $\mathscr{L}^f_p$ on $\Sigma_f$ interpolates the \textit{square roots} of \textit{central critical values} of the classical $L$-functions $L(f_x^0, g_y^0, h_z^0; s)$, which is a combined result of the Ichino formula, \cite[Theorem~4.7]{DR} and \cite[Corollary~5.13]{tripleL}.

Points in $\Sigma_{\bal}$, on the other hand, lie outside this region of interpolation, and we are interested in the values of $\mathscr{L}^f_p$ at these points.
By definition, the value of $\mathscr{L}^f_p (\mathbf{f}, \mathbf{g}, \mathbf{h})$ at a balanced point $(x, y, z)$ involves a `negative' power $\nabla^{-t}$, which is originally defined as a $p$-adic limit.
On the other hand, the interpolating property of $\nabla$ implies the relation $\nabla^s \circ \nabla^{-t} = \nabla^{-t+s}$.
In other words, this negative power also serves as the $t$-th anti-derivative with respect to $\nabla$.
Hence, we may expect to express these special values via a $p$-adic  integration theory.
A suitable candidate is Coleman $p$-adic integration or its generalization, Besser's finite polynomial cohomology.
Furthermore, these $p$-adic integration theories allow us to compute $p$-adic Abel--Jacobi maps (or Bloch--Kato logarithm maps), which will be discussed bellow.

Let $E \rightarrow X = X_1(N)$ be the universal generalized elliptic curve over the modular curve defined over $\bbQ_p$.
For any $n \geq 0$, we let $W_n$ be the $n$-th Kuga--Sato variety over $X$.
It is an $(n+1)$-dimensional variety obtained by desingularizing the $n$-fold fiber product $E^n$ over $X$ (cf. \cite[Appendix]{BDP}).

Now fix a balanced weight $(x, y, z) \in \Sigma_{\bal}$.
We will write $(x, y, z) = (r_1+2, r_2+2, r_3+2)$, $r := \frac{1}{2}(r_1+ r_2+ r_3)$ and set 
$$W := W_{r_1} \times W_{r_2} \times W_{r_3}.$$
As in \cite[\S~3.1]{DR}, one can construct a homologically trivial cycle
$$\Delta_{x, y, z} \in \CH^{r+2}(W)_0:= \ker(\CH^{r+2}(W) \xrightarrow{\cl} H^{2r+4}_{\dR}(W/ \bbC_p)),$$
and study its image under the $p$-adic Abel--Jacobi map
$$\AJ_p: \CH^{r+2}(W)_0 \rightarrow [\Fil^{r+2} H^{2r+3}_{\dR}(W/\bbC_p)]^\vee.$$

For any $\phi \in S_k(\Gamma_1(N) , \chi)$, we shall write its Hecke polynomial at $p$ as
$$ x^2 - a_p(\phi) x + \chi(p)p^{k-1} = (x- \alpha_\phi)(x-\beta_\phi)$$
with $\ord_p(\alpha_{\phi}) \leq \ord_p(\beta_\phi)$.
We say that $\phi$ is $p$-regular if $\alpha_\phi \neq \beta_\phi$.

Given the triple $\bff, \bfg, \bfh$, 
we are interested in their specializations $\bff_x, \bfg_y, \bfh_z$ that are $p$-stabilizations of classical modular forms $f_x^0, g_y^0, h_z^0$, respectively.
Moreover, we will assume throughout this article that $f_x^0, g_y^0, h_z^0$ are $p$-regular.
If $\bff$ is of slope $a$ with $a < x-1$, we also assume that $a = \ord_p(\alpha_{f_x^0})$.
To simplify the notations, we will write $f$ for ${f_x^0}$ and similarly for the modular forms $g$ and $h$ when no confusion might occur.

Associated with the triple  $(f_x^0, g_y^0, h_z^0)$, we construct an element
\begin{align*}
    \eta_f^a \otimes \omega_g \otimes \omega_h &\in H^{r_1+1}_{\dR}(W_{r_1}) \otimes \Fil^{r_2+1}H^{r_2+1}_{\dR}(W_{r_2}) \otimes \Fil^{r_3+1}H^{r_3+1}_{\dR}(W_{r_3}) \\
    & \subset \Fil^{r+2} ( H^{r_1+1}_{\dR}(W_{r_1}) \otimes H^{r_2+1}_{\dR}(W_{r_2}) \otimes H^{r_3+1}_{\dR}(W_{r_3})) \\
    & \subset \Fil^{r+2} H^{2r+3}_{\dR}(W),
\end{align*}
where the first inclusion is by the balancedness assumption, and the second one is from the K\"{u}nneth decomposition of $W$.
In particular, $\eta_f^a \otimes \omega_g \otimes \omega_h$ lies inside the domain of $\AJ_p(\Delta_{x, y, z})$.

The first result of this paper is the theorem below.
\begin{theorem}\label{theorem: main theorem in intro}
    Let $\bff, \bfg, \bfh$ be three finite slope families of modular forms and $\scrL_p^f(\bff, \bfg, \bfh)$ be the associated triple product $p$-adic $L$-function.
    Given $(x, y, z) \in \Sigma_{\bal}$.
    Let $c := (x+ y+ z- 2)/2$ and write $x = y +z -2t$ with $t \in \bbZ_{> 0}$.
    Then
    \begin{equation}
        \scrL_p^f(\bff, \bfg, \bfh)(x, y, z) =(-1)^{t-1} \frac{\scrE(f_x^0, g_y^0, h_z^0) }{(t-1)!\scrE_0(f_x^0)\scrE_1(f_x^0)} \AJ_p(\Delta_{x, y, z}) (\eta_f^a \otimes \omega_g \otimes \omega_h),
    \end{equation}
    where the Euler factors are given by
    \begin{align}
    \begin{split}
    \scrE_0(f_x^0) &:= 1 - \beta_f^2 \chi_f^{-1}(p) p^{1-x},\\
    \scrE_1(f_x^0) &:= 1-  \beta_f^2 \chi_f^{-1}(p) p^{-x},\\
    \scrE(f_x^0, g_y^0, h_z^0) &:= (1-\beta_f \alpha_g \alpha_h p^{-c}) (1-\beta_f \alpha_g \beta_h p^{-c}) (1-\beta_f \beta_g \alpha_h p^{-c}) (1-\beta_f \beta_g \beta_h p^{-c}).
    \end{split}
    \end{align}
\end{theorem}

\begin{remark} \label{remark: relation to BSV}
    We would like to remark that the Abel--Jacobi image appears in the above theorem has already been computed in \cite{BSV_BK_conjecture}, and it is related to the $p$-adic period
    \[ I_p(f, g, h) := \frac{(f^w_\alpha, e_{f^w_\alpha} \Xi_k(g, h) )_{Np} }{(f^w_\alpha,f^w_\alpha)_{Np}}
    \]
    defined in \textit{loc. cit.} which will be recall in \S \ref{subsection: p-adic period}.
    Proposition \ref{prop: specail values at balanced} of the current article relates the special value of $p$-adic $L$-function to this $p$-adic period.
    The theorem then follows.
\end{remark}

After establishing the $p$-adic Gross--Zagier formula, we proceed to generalize a few results in \cite{DR2}.

Let $E$ be an elliptic curve defined over $\bbQ$ and 
$$\rho : \Gal_{\bbQ} := \Gal (\overline{\bbQ} /\bbQ ) \rightarrow \Aut_L(V_\rho) \cong \GL_n (L)$$
be an Artin representation 
with coefficient in a finite extension $L$ of $\bbQ$, which factors through a finite Galois extension $H/ \bbQ$.
Then one has $L(E, \rho; s)$, the Hasse--Weil--Artin $L$-function of $E$ twisted by $\rho$. 

Let $E(H)^\rho := \Hom_{G_\bbQ}(V_\rho, E(H) \otimes_\bbZ L)$ be the $\rho$-isotypic component of the Mordell--Weil group.
The Galois-equivariant Birch and Swinnerton-Dyer conjecture then predicts the equality 
$$\ord_{s=1} L(E, \rho; s) = \dim_L (E(H))^\rho .$$

The triple product $L$-function appears in the special case $\rho = \rho_1 \otimes \rho_2$, where 
$$\rho_1, \rho_2 : G_{\bbQ} \rightarrow \GL_2(L)$$
are two-dimensional odd irreducible Artin representations with $\chi = \det(\rho_1) = \det(\rho_2)^{-1}$.
As the elliptic curve $E$ is associated with a cuspidal newform $f \in S_2(N_f)$, and the representations $\rho_1, \rho_2$ are associated with cuspidal newforms $g \in S_1(N_g, \chi), h \in S_1(N_h, \chi^{-1})$, respectively,
$L(E, \rho ;s)$ is equal to the triple product $L$-function $L(f, g, h; s)$.

Assuming that the modular forms $f, g, h$ are ordinary at a fixed prime $p$, and let $\mathbf{f}, \mathbf{g}, \mathbf{h}$ be three Hida families passing through respectively $f, g, h$, Darmon--Rotger proved the following result.
\begin{theorem} [{\cite[Theorem~B]{DR2}}] \label{theorem: Selmer group}
    If $L(f, g, h; 1)=0$ and $\scrL_p^g( \mathbf{f}, \mathbf{g}, \mathbf{h})(2, 1, 1) \neq 0$ for some choice of test vectors, then
    $$\dim_{L_p} \Sel_p( E, \rho) \geq 2$$
    where $L_p$ is the completion of $L$ in $\overline{\bbQ}_p$, $\Sel_p( E, \rho)$ is the $\rho$-isotypic component of the Bloch--Kato Selmer group (cf. \S 6.1 of \textit{loc. cit.}).
    The statement stays true if we replace $\scrL_p^g$ by $\scrL_p^h$.
\end{theorem}

The technique used in \cite{DR2} was later improved by Bertolini--Seveso--Venerucci in \cite{BSV_reciprocity_balanced}.
We will give an over-simplified illustration of their approach below.

First, one constructs an element (called the \textit{big diagonal class})
$$\underline{\kappa} = \underline{\kappa}(\mathbf{f}, \mathbf{g}, \mathbf{h}) \in H^1 (\bbQ, V(\mathbf{f}, \mathbf{g}, \mathbf{h}))$$
in the cohomology group of a big Galois representation $V(\mathbf{f}, \mathbf{g}, \mathbf{h})$, 
such that the image of $\underline{\kappa}_p :=res_p(\underline{\kappa})$ under the Perrin-Riou big logarithm 
$\scrL og_\xi$ is the triple product $p$-adic $L$-function $\scrL_p^\xi$ for $\xi \in \{f, g, h \}$, where $res_p$ denotes the restriction to $G_{\bbQ_p}$.

Second, one studies the specialization $\underline{\kappa}_p(\mathbf{f}_2, \mathbf{g}_1, \mathbf{h}_1)$ of the class $\underline{\kappa}_p(\mathbf{f}, \mathbf{g}, \mathbf{h})$ to weight $(2, 1, 1)$.
It turns out that the complex $L$-function encodes certain information about this specialization.
In particular, one has the following criterion.
\begin{theorem} [{\cite[Theorem~6.4]{DR} \& \cite[Theorem~B]{BSV_reciprocity_balanced}}]
    The class $\underline{\kappa}_p(\mathbf{f}_2, \mathbf{g}_1, \mathbf{h}_1)$ is crystalline if and only if the complex $L$-function $L(f, g, h; s)$ vanishes at the central point $s=1$.
\end{theorem}

\begin{remark}
    The above theorem actually holds for specializations at general unbalanced weights $(x, y, z)$, with the central point being $c = \frac{x+ y+ z- 2}{2}$.
\end{remark}

Lastly, one can construct, out of the crystalline class $\underline{\kappa}_p(\mathbf{f}_2, \mathbf{g}_1, \mathbf{h}_1)$, at least two linearly independent elements in $\Sel_p( E, \rho)$ and concludes the proof.

\begin{remark}
    Since $g, h$ are modular forms of weight $1$, they are necessarily ordinary at $p$.
    The two $p$-adic $L$-functions $\scrL_p^g$ and $\scrL_p^h$ are already defined in \cite{DR}.
    The remaining obstacles for the finite slope case is the existence of the class $\underline{\kappa}( \mathbf{f}, \mathbf{g}, \mathbf{h})$ that gives rise to $\scrL_p^f$,
    whose proof will rely on the $p$-adic Gross--Zagier formula for finite slope families.
\end{remark}

To summarize, our main result in the second part is the following theorem.

\begin{theorem}\label{theorem: main theorem 2 in intro}
    Let $(\bff, \bfg, \bfh)$ be as in the setting of Theorem \ref{theorem: main theorem in intro}.
    Then there is an element $\underline{\kappa}(\bff, \bfg, \bfh) \in H^1(\bbQ, V(\bff, \bfg, \bfh))$ such that 
    $$ \scrL og _f (\underline{\kappa}_p(\bff, \bfg, \bfh)) = \scrL_p^f(\bff, \bfg, \bfh).$$
    Moreover, when specialized to an unbalanced triple $(x, y, z) \in \Sigma_f$, the class $\underline{\kappa}_p(\bff_x, \bfg_y, \bfh_z)$ is crystalline if and only if the classical $L$-function $L(f_x^0, g_y^0, h_z^0 ; s)$ vanishes at the central point $c = \frac{x+y+z-2}{2}$.
\end{theorem}

\begin{remark}
    A one-variable class has already been constructed in \cite{DR2} and \cite{BSV_BK_conjecture}.
    Our three-variable class $\underline{\kappa}(\bff, \bfg, \bfh)$ can be viewed as a generalization of their constructions.
    In particular, one can use it to generalize \cite[Theorem~B]{BSV_reciprocity_balanced} in the case of exceptional zero.
    However, we prefer not to deal with the case of exceptional zero in this article.
\end{remark}

Most of the materials for proving Theorem \ref{theorem: main theorem 2 in intro} are already provided in \cite{BSV_reciprocity_balanced}, except for the various cohomology groups of ($\phig$)-modules and the big logarithm map for finite slope families.
Once these objects are introduced, the proof will follow identically the arguments in \cite[\S~8 and \S~9]{BSV_reciprocity_balanced}.
For this reason, we will mainly adopt their notations and omit most of the proofs in \S \ref{section: application}.

With Theorem \ref{theorem: main theorem 2 in intro}, it is not hard to see that one can generalize Theorem \ref{theorem: Selmer group} to the finite slope case, with the same strategy as in \cite[\S~6]{DR2}.



\vspace{3pt}
\noindent \textit{Acknowledgements}. 
The author would like to thank his supervisors Adrian Iovita and Giovanni Rosso for their constant support.
In addition, the author is grateful to Henri Darmon for his insightful advice and to Ananyo Kazi for his detailed explanation on the paper \cite{tripleL}.
Lastly, he appreciates the referee for feedbacks and suggestions.
The work of this article is funded by NSERC grants RGPIN 2018-04392, ALLRP 577144-22 and an FRQNT grant.

\section{Triple product \texorpdfstring{$p$}{}-adic \texorpdfstring{$L$}{}-functions} \label{section: L-function}

In this section, we will recall the definition of the triple product $p$-adic $L$-function for finite slope families.
The main reference is \cite{tripleL}, but we will follow \cite{katz-type} for a simpler setup.
Before going through the construction, we would like to briefly explain the idea of their work.

Fix an integer $N \geq 5$ and a prime $p \nmid N$.
Let $X = X_1(N)$ be the modular curve defined over $\bbZ_p$ with the universal generalized elliptic curve $\pi: E \rightarrow X$.
Let $Y \subset X$ be the affine modular curve and $C \subset X$ be the closed subscheme of cusps. 
Then the two sheaves
$$ \underline{\omega} := \pi_*( \Omega^1_{E/ Y}) \textrm{ and } \calH:= \mathbb{R}^1 \pi_* (\Omega^\bullet_{{E}/ Y})$$
on $Y$ have canonical extensions to $X$ (cf. \cite[\S~1]{BDP}), which will still be denoted by the same notations $\underline{\omega}$ and $\calH$.

For any $r \in \bbZ$, $H^0(X, \underline{\omega}^r)$ can be identified as the space of modular forms of weight $r$.
For a strict neighborhood $\calU$ of the ordinary locus in $X$, $H^0(\calU, \underline{\omega}^r)$ can be identified as the space of overconvergent modular forms of weight $r$.

The sheaf $\calH$ is equipped with a (increasing) filtration given by the exact sequence 
$$ 0 \rightarrow \underline{\omega} \rightarrow \calH \rightarrow \underline{\omega}^{-1} \rightarrow 0$$
and a Gauss--Manin connection
$$ \nabla: \calH \rightarrow \calH \otimes \Omega^1_X(\log C).$$
For any $r \in \bbZ$, we let $\calH^r := \Sym^r \calH$ be its $r$-th symmetric power.
The filtration and the connection on $\calH$ then extend naturally to $\calH^r$.
By composing with the Kodaira--Spencer isomorphism
$$\mathrm{KS}: \Omega^1_X(\log C) \xrightarrow{\sim} \underline{\omega}^2,$$
the Gauss--Manin connection may be viewed as a map from $\calH^r$ to $\calH^{r+2}$.

The idea is to construct a sheaf $\frakw_k$ (over a certain strict neighborhood) which $p$-adically interpolates the collection $\{ \underline{\omega}^r \}_{r \in \bbZ}$, such that a $p$-adic family of overconvergent modular forms may be viewed as a section of $\frakw_k$ over some strict neighborhood $\calU$ of the ordinary locus.

There is also a sheaf $\bbW_k$ which $p$-adically interpolates $(\calH^n, \Fil_\bullet, \nabla)_{n \in \bbZ}$.
In particular, it allows us to take certain $p$-adic powers of the connection $\nabla$, which is crucial for defining the triple product $p$-adic $L$-function $\scrL^f_p(\mathbf{f}, \mathbf{g}, \mathbf{h})$ as explained in the introduction.

Since their construction is highly technical, we prefer not to provide detailed proofs in this section.
Instead, we will only set up notations and statements that are sufficient for our application.

\subsection{The modular sheaf \texorpdfstring{$\frakw_k$}{} and the de Rham sheaf \texorpdfstring{$\bbW_k$}{} } \label{subsec: modular sheaf}

Let $N, p$ be as before and $X = X_1(N)$ be the modular curve defined over $\bbZ_p$.
We write $\frakX$ for the formal completion of $X$ along its special fiber.
We denote by $\Hdg$ the $\calO_\frakX$-ideal that is locally generated by a lift of the Hasse invariant.

For any integer $n \geq 2$, we denote by $\frakX_n$ the formal open subscheme in the formal admissible blow-up of $\frakX$ with respect to the ideal $\mathcal{J} = (p, \Hdg^n)$, defined by the condition that $\mathcal{J}$ is generated by $\Hdg^n$ over $\frakX_n$.
The generalized elliptic curve ${E} \rightarrow \frakX_n$ then admits a canonical subgroup of order $p^r$, where this $r$ depends on $n$ (cf. \cite[\S~3.1]{tripleL}).
We will write $\calX_n$ for the adic generic fiber of $\frakX_n$.





\begin{definition}
Let $\Lambda$ be a $p$-adically complete and separated ring, and $k: \bbZ_p^\times \rightarrow \Lambda^\times$ be a continuous homomorphism.
Then $k$ is called an analytic weight of radius $r \in \bbN$ if there exists an element $u = u_k \in p^{1-r} \Lambda$ such that $k(t) = \exp(u \log(t))$ for all $t \in 1 +p^r \bbZ_p$.
\end{definition}

\begin{definition}
    Let $k :\bbZ_p^\times \rightarrow \Lambda^\times$ be an analytic weight.
    By saying `specialized to a classical weight $\ell \in \bbZ_{>0}$,' it means that we take the composition $\tau \circ k $ for some homomorphism $\tau : \Lambda \rightarrow \bbZ_p$ such that $\tau \circ k : \bbZ_p^\times \rightarrow \bbZ_p^\times$ is the $\ell$-th power map.
    We will use the notation $\mid_{k \rightarrow \ell}$ to denote such specializations.
\end{definition}




One has the following results (cf. \cite[\S~3.2 \& \S~3.3]{tripleL}):

\begin{prop}
    Let $k$ be an analytic weight of radius $r$ as above, and take an integer $n$ such that ${E} \rightarrow \frakX_n$ admits a canonical subgroup of order $p^r$.
    Then there are $\calO_{\frakX_n} \hat{\otimes}_{\bbZ_p} \Lambda$-modules $\frakw_k$ and $\bbW_{k}$ with the following properties:
    \begin{enumerate}
        \item $\frakw_k$ is an invertible $\calO_{\frakX_n} \hat{\otimes}_{\bbZ_p} \Lambda$-module.
        Over the generic fiber $\calX_n$, the global sections of $\frakw_k$  can then be identified as $p$-adic families of overconvergent modular forms of weight $k$;
        \item $\bbW_k$ has a natural increasing filtration $(\Fil_i)_{i \geq 0}$ by $\calO_{\frakX_n} \hat{\otimes} \Lambda$-modules with $\Fil_0 = \frakw_k$. 
        Moreover, when specializied to a classical weight $\ell \geq 1$, $\Fil_\ell \bbW_\ell \cong \calH^\ell$ over $\calX_n$ and is compatible with the Hodge filtration of $\calH^\ell$.
        \item There is a Gauss--Manin connection $\nabla_{k}$ (with poles at $V(\Hdg)$) on $\bbW_{k}$ that satisfies Griffiths transversality.
        By composing with the Kodaira--Spencer map, the connection $\nabla_k$ can be viewed as a map from $\bbW_k$ to $\bbW_{k+2}$.
        When specialized to a classical weight $\ell \geq 1$, $\nabla_\ell$ is the classical Gauss--Manin connection $\nabla$ on $\calH^\ell$. 
        Namely, for $\bfg_k \in H^0(\frakX_n, \Fil_\ell \bbW_k$), $\nabla_k (\bfg_k) \mid_{k \rightarrow \ell} = \nabla \bfg_\ell$.
    \end{enumerate}

\end{prop}

\paragraph{The $q$-expansions}
As explained in \cite[\S~3.5]{tripleL}, we have the $q$-expansion map
$$H^0(\frakX_n, \bbW_k) \rightarrow \Lambda((q))$$
obtained by first restricting to the ordinary locus then applying the unit root splitting.
On the other hand, the $q$-expansion map can also be described by using the Tate curve. 

Let $T= \mathrm{Tate}(q^N)$ be the Tate curve over $\Spf(S)$, where $S = \Lambda ((q))$.
We fix the canonical basis $(\omega_{\can}, \eta_{\can} = \nabla(\theta)(\omega_{\can}))$ of $\cal H$ on $T$.
Let $\bbW_k(q)$ be the module of $\bbW_k$ over the Tate curve, \textit{i.e.,} the pull-back of $\bbW_k$ from $\frakX_n$ to $S$.
We have the following local description (cf. \cite[\S~3.5]{tripleL}): 
$$\bbW_k(q) = S \langle \frac{Y}{1+p^r Z} \rangle (1 + p^r Z)^k.$$
If we set $V_{k, i} := Y^i (1+ p^r Z)^{k-i}$, then the filtration of $\bbW_k$ is given by $\Fil_h \bbW_k(q) = \sum_{i=0}^h S \cdot V_{k,i}$.
The $q$-expansion map then corresponds to the map $\sum_i a_i(q) V_{k, i} \mapsto a_0(q)$.

\begin{remark} \label{remark: specializaion of q-expansion}
Since $k$ is analytic, the element $k(1 +p^rZ)$ is well-defined in $S \langle Z \rangle$ and we sometimes denote it by $(1 +p^r Z)^k$.
The description above comes directly from the explicit construction of vector bundles with marked sections (cf. \cite[\S~2]{tripleL}).
One should think of $1 +p^r Z$ (resp. $Y$) as an element that parametrizes the deformation of $\omega_{\can}$ (resp. $\eta_{\can})$.
In particular, the specialization of $V_{k, i}$ to a classical weight $\ell \in \bbZ_{> 0}$ can be identified as the element $\omega_{\can}^{\ell-i} \eta_{\can}^{i}$ of $\calH^\ell$ on the Tate curve when $i \leq \ell$.
\end{remark}

\paragraph{The operators $U$ and $V$}
Similar to the classical case, we can define Hecke operators $U$ on $H^0(\calX_n, \bbW_k)$ and $V: H^0(\frakX_{n}, \frakw_k) \rightarrow H^0(\frakX_{n+1}, \frakw_k)$, whose actions on the $q$-expansions are the usual ones
(cf. \cite[\S~3.6, \S~3.7]{tripleL}).
Moreover, $U$ is a compact operator.
In particular, one has the following proposition.
\begin{prop}[{\cite[Cor.~3.26]{tripleL}}]
    The operator $U$ on $H^0(\calX_n, \bbW_k)$ admits a Fredholm determinant $P(T) \in \Lambda \llbracket T \rrbracket$.
    In particular, for any non-negative rational number $a$, the space $H^0(\calX_n, \bbW_k)$ admits a slope $a$-decomposition
    $$H^0(\calX_n, \bbW_k) = H^0(\calX_n, \bbW_k)^{\leq a} \oplus H^0(\calX_n, \bbW_k)^{>a}.$$
\end{prop}

\begin{definition}
    Let $\bff \in H^0(\calX_n, \frakw_k)$. 
    We define $\bff^{[p]} := (1 - VU)\bff$ and call it the $p$-depletion of $\bff$.
    If $\bff(q) = \sum_{n=0}^\infty a_n(\bff) q^n V_{k, 0}$ is the $q$-expansion of $\bff$, then the $q$-expansion of $\bff^{[p]}$ is $\bff^{[p]}(q) = \sum_{n=0, \ p \nmid n}^\infty a_n(\bff) q^n V_{k, 0}$.
    Since $UV= 1$, it follows that $\bff^{[p]}$ lies in the kernel of $U$.
\end{definition}

\paragraph{$p$-adic iterations of the Gauss--Manin connection}

Another important property of the sheaf $\bbW_k$ is that it is possible to iterate the Gauss--Manin connection $\nabla$ to certain $p$-adic powers and can be described explicitly in terms of $q$-expansions.

We first need an assumption on the weights.

\begin{assumption} \label{assumption: iteration} (cf. \cite[Assumption~4.1]{tripleL})
Let $k$ and $s$ be two analytic weights for the same ring $\Lambda$ that satisfy
$$ k(t) = \chi(t) \omega(t)^a \exp( u_k \log(t)), \quad s(t) = \chi'(t) \omega(t)^b \exp ( u_s \log(t)) \quad \forall t \in \bbZ_p^\times$$
where 
\begin{enumerate}[i.]
    \item $\omega$ is the Teichm\"{u}ller character;
    \item $a, b \in \bbZ$, $a$ is even modulo $p$;
    \item $\chi$ and $\chi'$ are finite character of $\bbZ_p^\times$, and $\chi$ is even;
    \item $u_k \in p\Lambda$ and $u_s \in p^2 \Lambda$.
\end{enumerate}
\end{assumption}
Suppose that $k, s$ are two analytic weights of radius $r$ that satisfy Assumption \ref{assumption: iteration}.
Then the weight $k+2s$ is also analytic of radius $r$, and one can take $u_{k+2s} = u_k +2 u_s$.

We now state the following theorem on the $p$-adic iteration of $\nabla$.
\begin{theorem}  [{\cite[Theorem~4.3]{tripleL}}] \label{theorem: p-adic connection}
Suppose $\bfg \in H^0(\frakX_n, \bbW_k)^{U=0}$ and let $\bfg(q) = \sum_h \bfg_h(q) V_{k, h}$ be its $q$-expansion. 
Then there are positive integers $b$ and $\gamma$ depending on $k, s, n$ and an element
$\nabla^{s}\bfg \in \mathrm{Hdg}^{-\gamma} H^0(\frakX_b, \bbW_{k+2s})$ such that 
$$ \nabla^s(\bfg)(q) := \sum_h \sum_{j=0}^{\infty} \binom{u_s}{j} \prod_{i=0}^{j-1} (u_k + u_ s -h -1 -i) \theta^{s-j}\bfg_h(q) V_{k+2s, j+h},$$
where $\theta = q \frac{d}{dq}$.
\end{theorem}

\begin{remark}
    Let $\tau :\Lambda \rightarrow \bbZ_p$ be a homomorphism such that $\tau \circ s : \bbZ_p^\times \rightarrow \bbZ_p^\times$ is an integer $\ell$.
    Then the specialization of $\nabla^s(\bfg)$ via $\tau$ is $\nabla^\ell (\bfg)$, the usual $\ell$-th iteration of $\nabla$ (cf. \cite[Remark~4.4 (c)]{tripleL}, with the finite order character $\epsilon$ of $\bbZ_p^*$ in \textit{loc. cit.} taken to be trivial). 
\end{remark}
\begin{remark}
    The requirement that $u_k$ and $u_s$ are divisible by powers of $p$ is for convergence reason. 
    We would like to mention that an improved version is developed by Ananyo Kazi in his Ph.D. thesis \cite{twisted_tripleL_finite_slope}, where he constructed a slightly larger sheaf $\bbW_k'$ such that the $p$-adic iteration of $\nabla$ converges for $u_k, u_s \in \Lambda$.
    This is in accordance with the result of \cite{GPJ_iteration_GM_connection}.
    As such , we may assume that $u_k, u_s \in \Lambda$ in the remaining of this paper.
\end{remark}

\paragraph{The overconvergent projections}
We have seen how to define $\nabla^s\bfg^{[p]}$ for a suitable weight $s$. 
However, it is only a section of $\bbW_{k+2s}$ instead of $\frakw_{k +2s}$.
In order to define the triple product $p$-adic $L$-function via Petersson product of families of overconvergent modular forms, we need the overconvergent projection, which sends $\nabla^s\bfg^{[p]}$ to a certain section of $\frakw_{k+2s}$.

We first consider the following commutative diagram of sheaves on $\calX_n$
$$
\begin{tikzcd}
0 \arrow[r] &\Fil_m(\bbW_k) \arrow[d, "\nabla_k"] \arrow[r] &\bbW_k \arrow[d, "\nabla_k"] \arrow[r] & \bbW_k / \Fil_m (\bbW_k) \arrow[d, "\nabla_k"] \arrow[r] &0 \\
0 \arrow[r] &\Fil_{m+1}(\bbW_{k+2}) \arrow[r] &\bbW_{k+2} \arrow[r] & \bbW_{k+2} / \Fil_{m+1} (\bbW_{k+2}) \arrow[r] &0
\end{tikzcd}
$$
and view the columns as complexes $\Fil_m^\bullet \bbW_k$, $\bbW_k^\bullet$ and $(\bbW_k / \Fil^m)^\bullet$ respectively.

Note that the sheaves $\Fil_m \bbW_k$ and $\Fil_{m+1} \bbW_{k+2}$ are coherent, while the rest are not.
So the hypercohomology of $\Fil_m^\bullet \bbW_k$ on the Stein space $\calX_n$ can be computed by  the cohomology of the complex of global sections.
That is, for $i \geq 0$, we have
\begin{align*}
    H^i_{\dR}(\calX_n, \Fil_m^\bullet \bbW_k) :&= \mathbb{H}^i (\calX_n, \Fil_m^\bullet \bbW_k) \\
    &= H^i ( H^0(\calX_n, \Fil_m \bbW_k) \xrightarrow{\nabla_k} H^0(\calX_n, \Fil_{m+1} \bbW_{k+2})).
\end{align*}

\begin{theorem}[{\cite[Lemma~3.32 \& Definition~3.36]{tripleL}}]
    For any $m \in \bbN$, there is an isomorphism
    \begin{equation*}
        H^\dagger_m : H^1_{\dR}(\calX_n, \Fil_m^\bullet \bbW_k)\otimes_{\Lambda} \Lambda\left [ \prod_{i=0}^{m} (u_k -i) \right ] 
        \cong H^0(\calX_n, \frakw_{k+2})\otimes_{\Lambda} \Lambda\left [ \prod_{i=0}^{m} (u_k -i) \right ]
    \end{equation*}
    induced by the inclusion $H^0(\calX_n, \frakw_{k+2}) \rightarrow H^0(\calX_n, \Fil_{m+1} \bbW_{k+2})$.
\end{theorem}
\begin{theorem}[{\cite[Theorem~3.34 \& Definition~3.36]{tripleL}}] \label{theorem: oc projection in family}
    Fix a positive slope $a$.
    Then we have an integer $m_a \in \bbN$ and an isomorphism
    \begin{equation*}
        H^{\dagger, \leq a} : H^1_{\dR}(\calX_n, \bbW_k)^{\leq a} \otimes_{\Lambda} \Lambda \left [ \prod_{i=0}^{m_a} (u_k -i) \right ] 
        \cong H^0(\calX_n, \frakw_{k+2})^{\leq a} \otimes_{\Lambda} \Lambda \left [ \prod_{i=0}^{m_a} (u_k -i) \right ]
    \end{equation*}
    induced by the inclusion $H^0(\calX_n, \frakw_{k+2}) \rightarrow H^0(\calX_n, \Fil_{m_a+1} \bbW_{k+2})$,
    where the superscript $\leq a$ means the slope $\leq a$ subspace with respect to $U$.
\end{theorem}

The maps $H^\dagger_m$ and $H^{\dagger, \leq a}$ above are referred as the overconvergent projections.

\begin{remark}
    When specialized to a classical weight $\ell \in \bbZ_{>0}$, one can see that the map $H^\dagger_{m}$ coincides with the overconvergent projection defined in \cite[\S~3.2]{nocforms} when $m < \ell/2$.
\end{remark}

We shall now give the description of the map $H^\dagger_m$ on the $q$-expansions, which is an important input for later computations.
Recall that on $\bbW_k(q)$, we have the identity (cf. \cite[equation (7)]{tripleL})
\begin{equation*}
    \nabla_k (b V_{k, i}) = \theta(b) V_{k+2, i} + b(u_k -i) V_{k+2, i+1}, \quad b \in \Lambda((q)).
\end{equation*}

By definition, for $\gamma \in H^0(\calX_{n}, \Fil_{m+1}\bbW_{k+2})$, 
the overconvergent projection of a its $[ \gamma ]$ modulo $\nabla( H^0(\calX_{n}, \Fil_{m}\bbW_{k}))$ is the unique element in $H^0(\calX_{n}, \frakw_{k+2})$ that represents the same class.
Hence, after shifting the numbers, we have the following result.

\begin{prop}  [{\cite[Prop.~3.37]{tripleL}}] \label{prop: overconvergent projection}
Consider an element $\gamma \in H^0(\calX_{n}, \Fil_m \bbW_k)$ with its class $[\gamma ] \in  H^1_{\dR}(\calX_{n}, \Fil_{m-1}\bbW_{k-2}^\bullet)$.
Let $\gamma(q) = \sum_{i=0}^m \gamma_i(q) V_{k, i}$ be its $q$-expansion.
Then the $q$-expansion of $H^\dagger([\gamma]) \in H^0(\calX_n, \frakw_{k})$ is 
$$\sum_{i =0}^m (-1)^i \frac{\theta^i \gamma_i(q)}{(u_{k} -2 -i +1)(u_{k}-2 -i +2) \cdots (u_{k}-2)} V_{k,0}.$$
\end{prop}

\subsection{Definition of the triple product \texorpdfstring{$p$}{}-adic \texorpdfstring{$L$}{}-function}

We are now able to define triple product $p$-adic $L$-functions for finite slope families.
We first recall some results in the complex case.
For any newform $f$ of level $N_f$ and nebentypus $\chi_f$, we write $\bbQ_f$ for the number field generated by all Hecke eigenvalues of $f$.
We denote by $\pi_f$ the automorphic representation of $\GL_2(\mathbb{A}_\bbQ)$ generated by $f$.
If $N$ is a multiple of $N_f$ and $\bbQ_f \subset L$, then we let $S_k(N, L)[\pi_f]$ be the $f$-isotypic subspace of $S_k(N, L)$.

Now let $f \in S_k (N_f, \chi_f)$, $g \in S_\ell (N_g, \chi_g)$, $h \in S_m (N_h, \chi_h)$ be a triple of normalized primitive cuspidal eigenforms.
We set $N = \lcm(N_f, N_g, N_h)$ and $\bbQ_{f, g, h} = \bbQ_f \cdot \bbQ_g \cdot \bbQ_h$.
We assume that $N$ is square-free, $\chi_f \cdot \chi_g \cdot \chi_h =1$ and the triple $(k, \ell, m)$ is unbalanced with $k = \ell +m +2t'$ for some $t' \in \bbZ_{\geq 0}$.
In this situation, we have the following result of M. Harris and S. Kulda \cite{central-crit}, refined by A. Ichino \cite{trilinear}, T. C. Watson \cite{rankin-triple} and M.-L. Hsieh \cite{Hsieh_tripleL}, which is usual known as the Ichino formula.

\begin{theorem}\label{theorem: Ichino}
Let $f, g, h$ be as above, then there exist holomorphic modular forms
$$\breve{f} \in S_k(N, \bbQ_{f, g, h})[\pi_f], \ \breve{g} \in S_k(N, \bbQ_{f, g, h})[\pi_g], \ \breve{h} \in S_k(N, \bbQ_{f, g, h})[\pi_h]$$
and constants $C_q \in \bbQ_{f, g, h}$ depending only on the local components of $\breve{f}, \breve{g}, \breve{h}$ for all $q \mid N\infty$ such that
$$ \frac{\prod_q C_q}{\pi^{2k}} L \left ( f, g, h; \frac{k+l+m-2 }{2} \right ) = I(\breve{f}, \breve{g}, \breve{h})^2.$$ 
Moreover, there exists a choice of $\breve{f}, \breve{g}, \breve{h}$ such that all $C_q \neq 0$.
\end{theorem}
In the above theorem, $L(f, g, h; s)$ is the complex Garrett--Rankin triple product $L$-function and 
$$I(\breve{f}, \breve{g}, \breve{h}) := ( \breve{f}^*, \delta^t (\breve{g}) \cdot \breve{h} )$$
where $( \phantom{e} , \phantom{e})$ is the Petersson inner product, $\delta$ is the Maass--Shimura differential operator, and $\breve{f}^*= \breve{f} \otimes \chi_f^{-1}$ is the Atkin--Lehner involution of $\breve{f}$.


\paragraph{Triple product $p$-adic $L$-functions for finite slope families}

Let $f, g, h$ be as before and suppose that $f$ is of finite slope $a$ with $0 \leq a < k-1$ (cf. Remark \ref{reamrk: Petersson product of families}).
We fix modular forms $\breve{f}$, $\breve{g}$, $\breve{h}$ as in Theorem \ref{theorem: Ichino}.
We let $K/ \bbQ_p$ be a finite extension containing all the Fourier coefficients of $f, g, h$ and let $\calO_K$ be the ring of integers of $K$.

Let $\bff, \bfg, \bfh$ be overconvergent families of modular forms deforming the $p$-stabilizations of $f, g ,h$, respectively, and similarly $\breve{\bff}, \breve{\bfg}, \breve{\bfh}$ for $\breve{f}, \breve{g}, \breve{h}$.
We denote by 
\begin{align*}
    k_f: \bbZ_p^\times \rightarrow \Lambda_f, \quad
    k_g: \bbZ_p^\times \rightarrow \Lambda_g, \quad
    k_h: \bbZ_p^\times \rightarrow \Lambda_h
\end{align*}
the analytic weights of these families.
After base-change to $\calO_K$, we may assume that $\Lambda_f, \Lambda_g, \Lambda_h$ are $\calO_K$-algebras.
Then there exists some $n \in \bbN$ such that we have 
\begin{align*}
    \bff, \breve{\bff}  \in H^0(\calX_{n}, \frakw_{k_f}), \quad
    \bfg, \breve{\bfg}  \in H^0(\calX_{n}, \frakw_{k_g}), \quad
    \bfh, \breve{\bfh}  \in H^0(\calX_{n}, \frakw_{k_h}). 
\end{align*}


As we want to apply $p$-adic powers of $\nabla$ to the $p$-depletion $\breve{\bfg}^{[p]}$ later, we need the following assumption on the weights.

\begin{assumption} \label{assumption: nu triple product}
Suppose the weight $k_f - k_g - k_h$ is even, i.e., there is a weight $\nu: \bbZ_p^\times \rightarrow \Lambda_f \hat{\otimes}_{\calO_K} \Lambda_g \hat{\otimes}_{\calO_K} \Lambda_h$ such that $k_f - k_g - k_h =2\nu$.
Furthermore, $k_g$ and $\nu$ satisfy Assumption \ref{assumption: iteration}.
\end{assumption}

Under this assumption, we have
$$\nabla_{k_g}^\nu \breve{\bfg}^{[p]} \in H^0 ( \calX_{n'}, \bbW_{k_g +2 \nu})$$
for some $n' \geq n$.
Therefore, $\nabla_{k_g}^\nu \breve{\bfg}^{[p]} \times \breve{\bfh} \in H^0 ( \calX_{n'}, \bbW_{k_f})$ and we may consider its class in $H^1_{\dR}(\calX_{n'}, \bbW_{k_f-2})$.
After base change to $\mathfrak{K}_f$, we obtain a 
section in $H^1_{\dR}(\calX_{n'}, \bbW_{k_f-2}) \otimes_{\Lambda_f} \mathfrak{K}_f$,
where $\mathfrak{K}_f := \Frac(\Lambda_f)$.
We then consider its overconvergent projection 
$$H^{\dagger, \leq a} (\nabla_{k_g}^\nu \breve{\bfg}^{[p]}  \times \breve{\bfh} ) \in H^0(\calX_{n'}, \frakw_{k_f})^{\leq a} \otimes \mathfrak{K}_f.$$

To define the triple product $L$-function, we consider the Atkin-Lehner involution $\breve{\bff}^* = w_N(\breve{\bff}) \in H^0(\calX_{n}, \frakw_{k_f})$ (cf. \cite[Definition~5.2]{tripleL}).
At a classical point $x \in \Omega_{f, \cl}$, the specialization $\breve{\bff}^*_x$ the $p$-stabilization of the Atkin--Lehner involution of $\breve{\bff}_x^0$ of slope $a$.

\begin{definition}
The Garrett-Rankin triple product $p$-adic $L$-function attached to the triple $(\breve{\bff}, \breve{\bfg}, \breve{\bfh})$ of finite slope $p$-adic families of modular forms is

$$\mathscr{L}^f_p (\breve{\bff}, \breve{\bfg}, \breve{\bfh}) := \frac{( \breve{\bff}^{*} , H^{\dagger, \leq a} (\nabla_{k_g}^\nu \breve{\bfg}^{[p]}  \times \breve{\bfh} ) )}{( \breve{\bff}^{*} , \breve{\bff}^{*}  )} \in \mathfrak{K}_f \hat{\otimes} \Lambda_g \hat{\otimes} \Lambda_h$$
where the notation $( \phantom{e}, \phantom{e})$ will be explained in Remark \ref{reamrk: Petersson product of families} below. 
When $\bff, \bfg, \bfh$ are Hida families, this $p$-adic $L$-function coincides with the one given in \cite[\S~4.2]{DR}, which is an improvement of the result in \cite{Harris_Tilouine}.
\end{definition}

\begin{remark} \label{reamrk: Petersson product of families}
    Suppose $\bff$ is a family of normalized newforms of tame level $N_f$ and slope $a < \infty$.
    Let $M^{\leq a}$ be the space of families of slope $\leq a$ (e.g. $H^0(\mathcal{X}_{n'}, \mathfrak{w}_{k_f})^{\leq{a}}$).
    Then one may consider the map (cf. \cite[Proposition~4.5]{Hida_p-adic_measure1}, \cite[\S~3.3]{Hsieh_tripleL})
    \begin{equation*}
        1_{\bff}: M^{\leq a} \rightarrow \mathfrak{K}_f 
    \end{equation*}
    such that for any $\bfg \in M^{\leq a}$, $1_{\bff}(\bfg)$ interpolates the ratio of classical Petersson products of level $\Gamma_1(Np)$
    $$ \frac{(f_x, g_x)_{Np}}{(f_x, f_x)_{Np}}$$
    for classical weights $x > a+1$ with $f_x^0$ and $g_x^0$ being $p$-regular.
    In \cite[\S5]{tripleL}, a stronger condition 
    $x-1 > 2a$ is imposed, which then automatically implies the $p$-regularity.
    
    For a families of old forms $\breve{\bff}$, the functional $1_{\breve{\bff}}$ can also be constructed by the recipe in \cite[Lemma~2.19]{DR}.
    Consequently, the notation 
    $$\frac{( \breve{\bff}^{*} , H^{\dagger, \leq a} (\nabla_{k_g}^\nu \breve{\bfg}^{[p]}  \times \breve{\bfh} ) )}{( \breve{\bff}^{*} , \breve{\bff}^{*}  )}$$
    should be understood as $1_{\breve{\bff}} (H^{\dagger, \leq a} (\nabla_{k_g}^\nu \breve{\bfg}^{[p]} \times \breve{\bfh} ))$.
    Alternatively, one can use the Petersson product of families constructed in  \cite[\S~5.2.1]{tripleL}.
\end{remark}


\begin{prop}[{\cite[Corollary~5.13]{tripleL}}]
    For a unbalanced classical weight $(x, y, z)$ with $x$ being dominant, one has the interpolating formula
    $$\mathscr{L}^f_p (\breve{\bff}, \breve{\bfg}, \breve{\bfh})(x, y, z) = C \cdot I(\breve{f}_x, \breve{g}_y, \breve{h}_z)$$
    for some explicit Euler factor $C$.
\end{prop}


\subsection{Special values at balanced classical weights and \texorpdfstring{$q$}{}-expansions} \label{subsect: special values}

In this subsection, we specify at which classical points we want to study in the next section and fix several notations.

Suppose $\mathbf{f}$ is a family of finite slope $a= a_f$, and $\mathbf{g}, \mathbf{h}$ are of slopes $a_g, a_h$, respectively.
We are interested in the values of $\mathscr{L}^f_p (\mathbf{f}, \mathbf{g}, \mathbf{h})$ at classical weights $(x, y, z) \in \Sigma_{\cl}$ such that
\begin{enumerate}[i.]
    \item $x > a+1$, $y > a_g+1$ and $z > a_h+1$ (cf. Remark \ref{reamrk: Petersson product of families});
    \item The specializations of $\mathbf{f}, \mathbf{g}, \mathbf{h}$ at $(x, y, z)$ are $p$-old, \textit{i.e.}, $\mathbf{f}_x, \mathbf{g}_y, \mathbf{h}_z$ are $p$-stabilizations of some classical modular forms $f^0_x, g^0_y, h^0_z$ on $X_1(N)$. 
    Moreover, we require that $f^0_x, g^0_y, h^0_z$ are $p$-regular;
    \item The triple $(x, y, z)$ is balanced with $x = y+ z- 2t$ for some $t \in \bbZ_{>0}$.
\end{enumerate}
We denote the set consisting of all such weights by $\Sigma_{\mathbf{f}, \mathbf{g}, \mathbf{h}}$.
To simplify the notations, we will write $f_x := \mathbf{f}_x$, $g_y:= \mathbf{g}_y$, $h_z:= \mathbf{h}_z$ and denote the classical forms $f^0_x, g^0_y, h^0_z$ by $f, g, h$.
Notice that the new $f, g, h$ are different from the modular forms (of unbalanced weights $k, \ell, m$) which the families $\mathbf{f}, \mathbf{g}, \mathbf{h}$ deform respectively.
Let $c = c_{x,y,z} := \frac{x+ y+ z- 2}{2}$
We will use the notation $(x, y, z) = (r_1+2, r_2+2, r_3+2)$ freely from now on.

By definition of the $p$-adic $L$-function, we have
\begin{equation}
    \scrL_p^f (\mathbf{f}, \mathbf{g}, \mathbf{h})(x, y,z) := \frac{(f^*, H^{\dagger, \leq a}(\nabla^{-t}g^{[p]} \times h) )_{Np}}{(f^*, f^*)_{Np}}.
\end{equation}

The rest of this subsection is dedicated to computing the (polynomial) $q$-expansion of $H^{\dagger, \leq a}(\nabla^{-t}g^{[p]} \times h)$.

On the Tate curve, one has the canonical basis $\{ \omega_{\can}, \eta_{\can} \}$ of $\calH$ that satisfies
$$\nabla (\omega_{\can}) =\eta_{\can} \otimes \frac{dq}{q}, \quad \nabla(\eta_{\can} )=0.$$
As mentioned in Remark \ref{remark: specializaion of q-expansion}, the element $V_{r, i}$ of $\bbW_r(q)$ corresponds to the element $\omega_{\can}^{r-i} \eta_{\can}^i$ when $r$ is a positive integer and $i \leq r$.
From now on, we will make no difference between these two notations and use them interchangeably.

Let
$$g_y^{[p]}(q) V_{y, 0} \textrm{ and } h_z(q) V_{z, 0}$$
be the $q$-expansions of the specializations of $\mathbf{g}^{[p]}$ and $\mathbf{h}$ at $y$ and $z$ respectively.

Since we assume the triple $(x, y, z)$ to be balanced, $b:= y-t$ is a strictly positive integer.
Applying the formula in Theorem \ref{theorem: p-adic connection} and noticing that $-t$ is the specialization of $\nu$ at the weight corresponding to $(x, y, z)$ by Assumption \ref{assumption: nu triple product}, 
we see that
\begin{align*}
    \nabla^{-t} (g_y^{[p]} V_{y, 0}) &= \sum_j^\infty \binom{-t}{j} \prod_{i=0}^{j-1} (y- t- 1- i)\theta^{-t-j} g_y^{[p]}(q) V_{y-2t, j} \\
    &= \sum_{j=0}^{b-1} \binom{-t}{j} \prod_{i=0}^{j-1} (y- t- 1- i)\theta^{-t-j} g_y^{[p]}(q) V_{y-2t, j}
\end{align*}
is in fact a finite sum.
We simplify the above expression as 
$$\nabla^{-t} (g_y^{[p]} V_{y, 0}) = \sum_{j=0}^{b-1} C_j \cdot \theta^{-t-j} g_y^{[p]}(q) V_{y-2t, j}$$
by letting $C_j:=  \binom{-t}{j} \prod_{i=0}^{j-1} (y- t- 1- i)$. 

The $q$-expansion of the product $\nabla^{-t} g_y^{[p]}(q) \times h_z(q)$ can now be expressed as 
$$\sum_{j=0}^{b-1} C_j \cdot \theta^{-t-j} g_y^{[p]}(q) V_{y-2t, j} \times h_z(q) V_{z, 0} = \sum_{j=0}^{b-1} C_j \cdot \theta^{-t-j} g_y^{[p]}(q) \times h_z(q) V_{x, j}.$$
We here remark that $b-1 \leq x$.
Since if $b-1 = y-t-1 >x$, by using the balancedness assumption $x+z-y >0$, we have $z >t+1$.
But this would imply $x = y-t +z -t > y - t+1 > b-1$, which is a contradiction.
As a result, one has the identification
$$\sum_{j=0}^{b-1} C_j \cdot \theta^{-t-j} g_y^{[p]}(q) \times h_z(q) V_{x, j} = \sum_{j=0}^{b-1} C_j \cdot \theta^{-t-j} g_y^{[p]}(q) \times h_z(q) \omega_{\can}^{x-j} \eta_{\can}^j,$$
which can be viewed as a section of $\Fil_{b-1} \bbW_x \subset \Fil_{x} \bbW_x = \calH^x$.

Applying the overconvergent projection formula in Prop. \ref{prop: overconvergent projection} with $(m, k) = (b-1, x)$ and 
$$\gamma_j (q) = C_j \cdot \theta^{-t-j} g_y^{[p]}(q) \times h_z(q),$$
we get
\begin{align*}
    H^\dagger_m (\nabla^{-t}(g_y^{[p]} V_{y, 0}) \times h_z V_{z, 0})(q) = \sum_{j=0}^{b-1} (-1)^j C_j \cdot \frac{\theta^j (\theta^{-t-j} g_y^{[p]}(q) \times h_z(q))}{(x-2-j+1)(x-2 -j +2) \cdots (x-2)} V_{x, 0}.
\end{align*}
In the above identity, we adopt the convention that if $x-2-j+1 > x-2$ (that is, when $j=0$), then the product in the denominator is understood to be $1$.

\section{The \texorpdfstring{$p$}{}-adic Gross--Zagier formula} \label{section: p-adic GZ formula}

In this section, we will prove the $p$-adic Gross--Zagier formula.
As mentioned in Remark \ref{remark: relation to BSV}, we will show the relation between special values of $\mathscr{L}^f_p$ at $\Sigma_{\mathbf{f}, \mathbf{g}, \mathbf{h}}$ and the $p$-adic periods $I_p(f, g, h)$ in \cite{BSV_BK_conjecture}.
We will use freely the notations introduced in the previous section.
For simplicity, we will also drop the $\breve{\phantom{e}}$ notation throughout this section.

\subsection{The \texorpdfstring{$p$}{}-adic period \texorpdfstring{$I_p(f, g, h)$}{}} \label{subsection: p-adic period}

We briefly recall the definition of the $p$-adic period $I_p(f, g, h)$ in \cite{BSV_BK_conjecture}.
Throughout this subsection, we will fix a triple $(x, y, z) \in \Sigma_{\mathbf{f}, \mathbf{g}, \mathbf{h}}$ and drop the notations $x, y, z$ if no confusion might occur.

Let $X^{ord} = X^{ord}_1$ denotes the ordinary locus of $X_1(N)$, viewed as an affinoid.
We first observe that, since we assume that $a = a_f <x-1$, there is a sequence of isomorphisms
\begin{align} \label{equation: oc and slope projection}
\begin{split}
    H^0(X_1(pN), \omega^{r_1+2})&^{\leq a} \cong H^0(X^{ord, \dagger}, \omega^{r_1+2} )^{\leq a} \\
    &\cong \left ( \frac{H^0(X^{ord, \dagger}, \calH^{r_1} \otimes \Omega^1 )}{ \nabla H^0(X^{ord, \dagger}, \calH^{r_1} ) } \right)^{\leq a} \cong H^1_{\rig} (X^{ord}, \calH^{r_1})^{\leq a}
\end{split}
\end{align}
by Coleman's classicality (cf. \cite{coc_mf}).
Here, by $H^0(X^{ord, \dagger}, -)$ we mean the subspace of overconvergent sections in $H^0(X^{ord}, -)$.

\begin{remark}\label{remark: noc section}
    The element $\theta^{-t} g^{[p]} \times h$, who plays an important role throughout \S 3, is a nearly overconvergent modular forms.
    Namely, it is the image under the unit root splitting of a unique element in $H^0(X^{ord, \dagger}, \calH^{r_1} \otimes \Omega^1)$.
    In fact, we have
    $$ \theta^{-t} g^{[p]} \times h = \mathrm{Spl}_{\textrm{u-r}} (\pr_{r_1}(G^{[p]} \times \omega_h)),$$
    where 
    $$\mathrm{Spl}_{\textrm{u-r}}: H^0(X^{ord, \dagger}, \calH^{r_1} \otimes \Omega^1) \rightarrow H^0(X^{ord}, \calH^{r_1} \otimes \Omega^1) \rightarrow H^0(X^{ord}, \omega^{r_1} \otimes \Omega^1)$$
    denotes the composition of the restriction and the unit-root splitting (the map $\mathrm{Spl}_{\textrm{u-r}}$ is known to be injective), and the element $\pr_{r_1}(G^{[p]} \times \omega_h)$ will be introduced in the next section.
    For a proof of this fact, one could refer himself to \cite[Proposition 2.9]{DR}, \cite[(46)]{BSV_BK_conjecture} or \cite[Remark 6.3]{tripleL_and_GZ_formula}.
    
\end{remark}

Let $f^* := w_N(f)$ (which is denoted by $f^w$ in \cite{BSV_BK_conjecture}).
We let $f^*_\alpha$ be the $p$-stabilization of $f^*$ such that the $U$-eigenvalue is $\alpha_{f^*}$, with $\ord_p(\alpha_{f^*}) = a < x-1$.

\begin{definition}[cf. {\cite{BSV_BK_conjecture}}]
    The $p$-adic period $I_p(f, g, h)$ is defined to be 
    \begin{equation}
        I_p(f, g, h) := \frac{(f^*_\alpha,  e^{\leq a} (\theta^{-t} g^{[p]} \times h) )_{Np}}{(f^*_\alpha, f^*_\alpha)_{Np}}. 
    \end{equation}
    Notice that this is an abuse of notation. 
    By Remark \ref{remark: noc section}, we can consider the image of $e^{\leq a} (\theta^{-t} g^{[p]} \times h)$ in $H^1_{\rig} (X^{ord}, \calH^{r_1})^{\leq a}$ and identify it as an element $\Xi_k^{\leq a}(g, h)$ in $H^0(X_1(pN), \omega^{r_1+2})^{\leq a}$ via (\ref{equation: oc and slope projection}).
    Then ${(f^*_\alpha,  e^{\leq a} (\theta^{-t} g^{[p]} \times h) )_{Np}}$ is understood as 
    $(f^*_\alpha,\Xi_k^{\leq a}(g, h) )_{Np}$.
\end{definition}

\begin{prop} \label{prop: specail values at balanced}
If $(x, y, z)$ is a balanced triple and $(f, g, h)$ denotes the specialization of $(\mathbf{f}, \mathbf{g}, \mathbf{h})$ at $(x, y, z)$,
then we have
\begin{equation}
    \scrL_p^f (\mathbf{f}, \mathbf{g}, \mathbf{h})(x, y,z) = I_p(f, g, h).
\end{equation}
\end{prop}

\begin{proof}
    By the definition of $\scrL_p^f (\mathbf{f}, \mathbf{g}, \mathbf{h})$, one has
    \[
    \scrL_p^f (\mathbf{f}, \mathbf{g}, \mathbf{h})(x, y,z) := \frac{(f^*_\alpha, H^{\dagger, \leq a}(\nabla^{-t}g^{[p]} \times h) )_{Np}}{(f^*_\alpha, f^*_\alpha)_{Np}}.
    \]
    Notice that $\nabla^{-t}g^{[p]} \times h$ is nearly overconvergent whose unit-root splitting equal to $\theta^{-t} g^{[p]} \times h$, by the explicit expression in \S 2.3.
    The Petersson product is again understood as viewing $H^{\dagger, \leq a}(\nabla^{-t}g^{[p]} \times h)$ as an element in $H^0(X_1(pN), \omega^{r_1+2})^{\leq a}$ via the first isomorphism of (\ref{equation: oc and slope projection}).

    Also observe that the isomorphism 
    \[ \left ( \frac{H^0(X^{ord, \dagger}, \calH^{r_1} \otimes \Omega^1 )}{ \nabla H^0(X^{ord, \dagger}, \calH^{r_1} ) } \right)^{\leq a} \cong H^1_{\rig} (X^{ord}, \calH^{r_1})^{\leq a}\]
    in (\ref{equation: oc and slope projection})
    coincides with $H^{\dagger, \leq a}$, which can be seen from Theorem \ref{theorem: oc projection in family} as they are both induced by the Hodge filtration. 

    Hence, we deduce that 
    $ H^{\dagger, \leq a}(\nabla^{-t}g^{[p]} \times h) =  e^{\leq a} (\theta^{-t} g^{[p]} \times h) $ (again an abuse of notation) and the result follows.
\end{proof}


\subsection{The \texorpdfstring{$p$}{}-adic Abel--Jacobi images of generalized diagonal cycles} \label{subsection: AJ imgae}
This subsection is dedicated to computing the $p$-adic Abel--Jacobi map. 

\paragraph{Kuga--Sato varieties and the generalized diagonal cycles}

For any $r \geq 0$, we have the Kuga--Sato variety $W_r$, which is a desingularization (cf. \cite[Appendix]{BDP}) of the $r$-fold fiber product 
$$W_r' := E \times _X E \cdots \times_X E.$$

Then one may see the parabolic cohomolgy $H^1_{\para} (X, \calH^r)$ as a subspace in a correct degree of the de Rham cohomology of $W_r$.
This is illustrated in the following lemma.

\begin{lemma}[{\cite[Lemma~2.2]{BDP}}]
Assume that $r \geq 1$. 
Then there exists an idempotent $\epsilon_r \in \bbQ[ \Aut(W_r/X)]$ such that
\begin{equation*}
    \epsilon_r H^j_{\dR}(W_r/ K) = \begin{cases} 0  &j \neq r+1 \\
    H^1_{\para} (X, \calH^r) &j=r+1
    \end{cases}.
\end{equation*}
Furthermore, the Hodge filtration on $\epsilon_r H^{r+1}_{\dR}(W_r/ K)$ is given by
\begin{align*}
    &\Fil^0 = H^1_{\para} (X, \calH^r), \\
    &\Fil^1 = \Fil^2 = \cdots = \Fil^{r+1} = H^0(X, \underline{\omega}^r \otimes \Omega^1_X), \\
    & \Fil^{r+2} =0.
\end{align*}
\end{lemma}

\begin{remark}
    We here recall that the parabolic cohomology $H^1_{\para} (X, \calH^r)$ is defined to be the hypercohomology of the complex
    $$0 \rightarrow \calH^r \rightarrow \left ( \calH^r \otimes \Omega^1_X + \nabla(\calH^r) \right ) \rightarrow 0.$$
    As we are mainly interested in cusp forms, the parabolic cohomology has the advantage that it comes with the exact sequence (cf. \cite[\S~2]{mf_deRham})
    $$ 0 \rightarrow H^0(X, \underline{\omega}^r \otimes \Omega^1_X) \rightarrow H^1_{\para} (X, \calH^r) \rightarrow H^1(X, \underline{\omega}^{-r}) \rightarrow 0.$$
\end{remark}

Suppose now we have a triple of balanced weights $(x, y, z) = (r_1+2, r_2 +2, r_3+2)$.
We further assume that $r_1 > 0, r_2 > 0, r_3 >0$ and $r := \frac{r_1+r_2+r_3}{2} \in \bbN$.

Set $W = W_{r_1} \times_K W_{r_2} \times_K W_{r_3}$, which is of dimension $2r+3$.
We now briefly recall the definition of the generalized diagonal cycle $\Delta_{x, y, z} \in \mathrm{CH}^{r+2}(W)$ (cf. \cite[Definition~3.3]{DR}).

Choose three subsets
$$A = \{ a_1, \ldots, a_{r_1} \}, \quad B = \{ b_1, \ldots, b_{r_2} \}, \quad C = \{ c_1, \ldots, c_{r_3} \}$$
of $\{1, \ldots, r\}$ such that $A \cap B \cap C = \emptyset$ and $A \cup B = B \cup C = A \cup C = \{1, \ldots, r\}$.
One can see that the balancedness assumption guarantees the existence of such sets.
Then we consider the map
\begin{align*}
     \varphi_{ABC}: W_r &\rightarrow W:= W_{r_1} \times_K W_{r_2} \times_K W_{r_3}\\
      (x; P_1, \ldots P_r) &\mapsto ((x;P_{a_1}, \ldots, P_{a_{r_1}}), (x;P_{b_1}, \ldots P_{b_{r_2}}), (x;P_{c_1}, \ldots P_{c_{r_3}})),
\end{align*}
which is a closed embedding of $W_r$ into $W$.

\begin{definition}
The generalized diagonal cycle is defined by
$$\Delta_{x, y, z} = (\epsilon_{r_1}, \epsilon_{r_2}, \epsilon_{r_3}) \varphi_{ABC}(W_r) \in \mathrm{CH}^{r+2}(W).$$

\end{definition}

\begin{remark}
    Here we assume that $x, y, z >2$ for simplicity.
    The definition of $\Delta_{x, y, z}$ when some of the weights are equal to $2$ can be found in \cite[\S~3]{DR}.
\end{remark}

By using the K\"{u}nneth decomposition on $H^{2r+4}_{\dR}(W)$ and examining the image of the idempotents $\epsilon_{r_i}$'s, it follows that the cycle $\Delta_{x, y, z} $ is homologous to zero.
That is,
$$ \Delta_{x, y, z} \in \mathrm{CH}^{r+2}(W)_0 := \ker (\cl : \mathrm{CH}^{r+2}(W) \rightarrow H^{2r+4}_{\dR}(W) ). $$
In particular, one may consider the image of $\Delta_{x, y, z}$ 
under the $p$-adic Abel--Jacobi map (cf. \cite{p-AJmap})
$$\AJ_p : \mathrm{CH}^{r+2}(W)_0 \rightarrow \Fil^{r+2}H^{2r+3}_{\dR}(W)^\vee.$$

\begin{remark}
    Alternatively, one can consider the diagonal class $$\tilde{\kappa}_{N, \mathbf{r}} \in H^1 (\bbQ, \texttt{W}_{N, \mathbf{r}} )$$
    defined in \cite[\S 3]{BSV_reciprocity_balanced}, whose relation with the above cycle $\Delta_{x, y, z}$ is provided in Proposition 3.7 of \textit{loc. cit.}.
\end{remark}

\paragraph{Explicit computations of the Abel--Jacobi maps}

We now study the value of $\AJ_p(\Delta_{x, y, z})$ at a certain element $\eta \otimes \omega_2 \otimes \omega_3$ in
$ H^1_{\para}(X, \calH^{r_1}) \otimes H^0(X, \underline{\omega}^{r_2} \otimes \Omega^1_X) \otimes  H^0(X, \underline{\omega}^{r_3} \otimes \Omega^1_X).$

First of all, by the balancedness assumption, we see that 
$$H^1_{\para}(X, \calH^{r_1}) \otimes H^0(X, \underline{\omega}^{r_2} \otimes \Omega^1_X) \otimes  H^0(X, \underline{\omega}^{r_3} \otimes \Omega^1_X) \subset \Fil^{r+2}H^{2r+3}_{\dR}(W).$$



We now specify the forms $\eta, \omega_2, \omega_3$ we will encounter later.
Let $f, g, h$ be a triple of cusp forms on $X_1(N)$ of balanced weights $(x =r_1+2,\ y =r_2+2,\ z=r_3+2)$ such that $x = y+z -2t$.
The cusp form $f$ then corresponds to a section $\omega_f \in H^0(X, \underline{\omega}^{r_1} \otimes \Omega^1_X)$ and similar for $g$ and $h$.
We also reserve the notation $\calU$ for an auxiliary strict neighborhood of the ordinary locus in $X$ (which may vary according to the situation) and view $\omega_g, \omega_h$ as sections over $\calU$ via restriction.

We write the Hecke polynomial of $f$ as 
$$T^2 -a_p(f) T +\chi_f(p) p^{x-1} = (T-\alpha_f)(T-\beta_f)$$
with $a= \ord_p(\alpha_f)$, and similarly for $g, h$.
We also write
\begin{align*}
    P_g(T) &= 1- a_p(g) p^{1- y}T + \chi_g(p) p^{1-y} T^2 = (1- \alpha_g p^{1-y}T)(1- \beta_g p^{1-y}T), \\
    P_h(T) &= 1- a_p(h) p^{1- z}T + \chi_h(p) p^{1- z} T^2 = (1- \alpha_h p^{1-z}T)(1- \beta_h p^{1-z}T).
\end{align*}
We will write
$\alpha_g' := \alpha_g p^{1-y}$, $\beta_g' := \beta_g p^{1-y}$, $\alpha_h' :=\alpha_h p^{1-z}$ and $\beta_h' := \beta_h p^{1-z}$.
The polynomial $P_g(T)$ is defined such that $P_g(\phi)$ annihilates the class of $\omega_g$, where $\phi$ is the Frobenius on the cohomology $H^1_{\para}(X, \calH^{r_2})$.
In fact, $P_g(\phi) \omega_g = \omega_{g^{[p]}}$ as section over $\calU$.

By \cite[Cor.~2.13]{DR} and the paragraph before it, we have a class $\overline{\eta_f} \in H^1(X, \underline{\omega}^{-r_1})$
such that for any cusp form $\mu \in S_x(\Gamma_1(N))$, we have
$$\langle \overline{\eta_f}, \omega_\mu \rangle_{\dR} = \frac{( f^*, \omega_\mu )_N}{( f^*, f^* )_N}.$$

We want to find a suitable lift of $\overline{\eta_f}$ in $H^1_{\para}(X_K, \calH^{r_1})$.
First recall that Poincar\'{e} pairing on $H^1_{\para}(X_K, \calH^{r_1})$ descends to a perfect pairing
\begin{equation*}
    H^1_{\para}(X_K, \calH^{r_1})[f] /\Fil^{r_1+1} \times S_k(N; K)[{f}^*] \rightarrow K.
\end{equation*}
Hence, one can viewed $\overline{\eta_f}$ as an element in $H^1_{\para}(X_K, \calH^r)[f] /\Fil^{r+1}$.
We here remark that on the $f$-isotypic part $H^1_{\para}(X_K, \calH^{r_1})[f]$, the Frobenius $\phi$ acts with eigenvalues $\alpha_f/ p^{x-1} = \beta_{f^*}$ and $\beta_f/ p^{x-1} = \alpha_{f^*}$.
Notice that $\alpha_{f^*} = \chi_f(p)^{-1} \alpha_f$ is also of $p$-adic valuation $a$.

Suppose that $a <x-1$, then one has the following decomposition (cf. \cite{Saito_p_Hodge} or \cite[\S~2.4.3]{familyofGalois})
$$ H^1_{\para}(X, \calH^{r_1})[f] =  \Fil^{r_1+1} H^1_{\para}(X, \calH^{r_1}) [f] \oplus H^1_{\para}(X, \calH^{r_1})[f]^{\phi = \alpha_{f^*}}.$$
This splitting allows us to fix a lift 
$\eta_f =\eta_f^a \in H^1_{\para}(X, \calH^{r_1})[f]^{\phi = \alpha_{f^*}}$ of $\overline{\eta_f}$.


We now turn our attention to the modular forms $g$ and $h$.
As $1-VU$ annihilates the cohomology group $H^1_{\para}(X, \calH^{r_2})$, the form $\omega_g^{[p]}$ over $\calU$ is $\nabla$-exact.
A primitive $G^{[p]}$ over $\calU$ can be chosen (uniquely) such that its polynomial $q$-expansion takes the form
$$G^{[p]}(q) = \sum_{i=0}^{r_2} (-1)^i i! \binom{r_2}{i} \theta ^{-i-1} g^{[p]} (q) \omega_{\can}^{r_2-i} \eta_{\can}^i .$$
The product $G^{[p]} \times \omega_h$ can then be viewed as a section of $\calH^{r_2} \otimes \calH^{r_3} \otimes \Omega_X^1$.
By assumption on the weights $(x, y, z)$, it follows that there is a projection $\pr_{r_1}: \calH^{r_2} \otimes \calH^{r_3} \rightarrow \calH^{r_1}(1-t)$ (cf. \cite[Prop.~2.9]{DR}).
As a result, one has a section $\pr_{r_1}(G^{[p]} \times \omega_h)$ of $\calH^{r_1}(1-t) \otimes \Omega^1_X$.

Now we are able to describe the Abel--Jacobi image.
\begin{theorem}\label{theorem: AJ map of diagonal cycle}
Let $\Delta_{x, y, z}$ be the general diagonal cycle defined as above, and view $\eta_f \otimes \omega_g \otimes \omega_h$ as elements in the de Rham cohomology $\Fil^{r+2} H^{2r+3}_{\dR}(W)$.
Then we have
\begin{align*}
    \AJ_p(\Delta_{x, y, z})(\eta_f \otimes \omega_g \otimes \omega_h) &= \frac{\scrE_1(f)}{\scrE(f, g, h)} \cdot \langle \eta_f,  e^{\leq a} \pr_{r_1}(G^{[p]} \times \omega_h) \rangle_{\dR} \\
    &= (-1)^{t-1} \frac{(t-1)! \scrE_0(f) \scrE_1(f)}{\scrE(f, g, h)} \times I_p(f, g, h),
\end{align*}
where $\langle \phantom{e}, \phantom{e} \rangle_{\dR}$ is the Poincar\'{e} pairing.

The terms $\scrE_0(f), \scrE_1(f)$ and $\scrE(f, g, h)$ are Euler factors defined by
\begin{align*}
    \scrE_0(f) &= (1 -\beta_{f}^2 \chi_{f}^{-1}(p) p ^{1-x}), \\
    \scrE_1(f) &=  (1- \beta^2 \alpha_g' \beta_g' \alpha_h' \beta_h') = (1- \beta_f^2 \chi_f^{-1}(p) p^{-x}),\\
    \scrE(f, g, h) &=  (1-\beta \alpha_g' \alpha_h') (1-\beta \alpha_g' \beta_h') (1-\beta \beta_g' \alpha_h') (1-\beta \beta_g' \beta_h'), 
\end{align*}
where $\beta := p^{x-1+t-1}/ \alpha_{f^*}$.
\end{theorem}

\begin{proof}
    The fist equation follows from \cite[Theorem 3.14]{DR} for the ordinary case and \cite[Theorem A]{BSV_BK_conjecture} for the non-ordinary case.
    The second equation follows from \cite[Theorem A \& (33)]{BSV_BK_conjecture}.
\end{proof}

\subsection{The \texorpdfstring{$p$}{}-adic Gross--Zagier formula}

We first recall our notations.
Let $\mathbf{f}, \mathbf{g}, \mathbf{h}$ be three finite slope families and $(x, y, x) \in \Sigma_{\mathbf{f}, \mathbf{g}, \mathbf{h}}$ be a balanced triple of weights as in Section \ref{subsect: special values}.
The specializations $f_x, g_y, h_z$ are classical modular forms on $X_1(N, p)$ which are the respective $p$-stabilizations of modular forms $f, g, h$ on $X_1(N)$.
We also have the elements $\eta_{f}, \omega_{g}, \omega_{h}$ associated with $f, g, h$ in their respective cohomology groups as defined in Section \ref{subsection: AJ imgae}.

\begin{theorem}[$p$-adic Gross--Zagier formula] \label{theorem: main theoerem}
Let notations be as above.
We have
\begin{equation*}
    \scrL_p^f (\mathbf{f}, \mathbf{g}, \mathbf{h})(x, y,z) = (-1)^{t-1} \frac{\scrE(f, g, h)}{(t-1)! \scrE_0(f) \scrE_1(f)} \times \AJ_p(\Delta_{x, y,z})(\eta_f \otimes \omega_g \otimes \omega_h) 
\end{equation*}
\end{theorem}

\begin{proof}
    This follows immediately from Proposition \ref{prop: specail values at balanced} and Theorem \ref{theorem: AJ map of diagonal cycle}.
\end{proof}

\section{An application to the Bloch--Kato conjecture} \label{section: application}

In this section, we will show an application of the $p$-adic Gross--Zagier formula for finite slope families.
Our goal is to prove Theorem \ref{theorem: main theorem 2 in intro} in the introduction.


\subsection{Coleman families and big Galois representations} \label{subsection: Coleman families}
Let $\calW$ be the weight space defined over $\bbQ_p$ as before.
In this section, we will focus on families defined over wide open disks in $\calW$ as in \cite[\S~4]{LZ-Rankin_Eisenstein}.

\begin{definition} \label{definition: open disk}
    Suppose $\calV$ is a open disk in $\calW$ defined over a finite extension $L$ over $\bbQ_p$.
    We denote by $\scrO_{\calV}$ (resp. $\Lambda_{\calV}$) the ring of bounded (resp. bounded by $1$) rigid analytic functions on $\calV$.
    In particular, if $k_0 \in \bbZ$ is a center of $\calV$,     
    $\Lambda_\calV \cong \calO_L \llbracket \bm{k}- k_0 \rrbracket$ for some formal variable $\bm{k}$ and $\scrO_\calV = \Lambda_\calV [\frac{1}{p}]$.
    We will always endow $\Lambda_\calV$ with the topology induced by its maximal ideal $m_{\calV}$.
    We let $\kappa_\calV : \bbZ_p^\times \rightarrow \Lambda_\calV^\times$ be the universal character (shifted by $-2$) on $\calV$.
    In other words, $\kappa_\calV (t) = \omega(t)^{k_0-2} \langle t \rangle^{\mathbf{k}-2}$ for $t \in \bbZ_p^\times$, where $\omega$ is the Teichm\"{u}ller character and $\langle t \rangle := t/ \omega(t) \in 1 +p \bbZ_p$.
\end{definition}

\begin{definition}
    Let $\calV \subset \calW$ be an open disk such that the classical weights $\calV \cap \bbZ_{\geq 0}$ are dense in $\calV$.
    A Coleman family $\bff$ (of level $N_f$ and nebentypus character $\chi_f$) defined over $\calV$ of slope $a \in \bbQ_{\geq 0}$ is a power series 
    $$ \bff = \sum_{n \geq 1} a_n(\bff) q^n \in \Lambda_\calV \llbracket q \rrbracket$$
    such that $a_1(\bff) =1$, $a_p(\bff)$ is invertible in $\scrO_{\calV}$ with norm $p^{-a}$, and for all but finitely many classical weights $k \in \calV$, the specialization $\bff_k \in \calO_L \llbracket q \rrbracket$ is the $q$-expansion of a classical normalized eigenform of weight $k$, level $\Gamma_1(N_f, p)$ and nebentypus $\chi_f$.
    We will often write $\calV_f = \calV$ and similarly $\scrO_f = \scrO_{\calV_f}$, $\Lambda_f = \Lambda_{\calV_f}$, $\kappa_f = \kappa_{\calV_f}$ to specify their dependence to the family $\bff$.
\end{definition}

One now wants to construct a $p$-adic family of Galois representations attached to a finite slope family.
It turns out that this can be achieved by considering certain big \'{e}tale sheaf on modular curves (cf. \cite{Overconvergent_ES} and \cite{p-adic_deformation}).
We shall briefly recall several notations and results without any proof, mainly following the setups in \cite[\S~4]{BSV_reciprocity_balanced}.

As we will consider both families and their specializations to classical weights,
we use the terminology `big' to refer to anything defined for families and `small' for the specialized case.

\paragraph{Locally analytic functions and distributions}

\begin{definition}
    For any integer $m \geq 0$, we write $LA_m(\bbZ_p, \Lambda_\calV)$ for the space of functions $\bbZ_p \rightarrow \Lambda_{\calV}$ that are analytic on every open disk $z + p^m \bbZ_p$ of $\bbZ_p$.
\end{definition}

We define $T := \bbZ_p^\times \times \bbZ_p$ and $T' = p\bbZ_p \times \bbZ_p^\times$.
One has right actions of $M_{2 \times 2}(\bbZ_p)$ on $T$ and $T'$.
An easy calculation shows that $T$ is preserved by 
$$ \Sigma_0(p) = \begin{pmatrix}
    \bbZ_p^\times & \bbZ_p \\ p \bbZ_p & \bbZ_p
\end{pmatrix}$$
and  $T'$ is preserved by 
$$ \Sigma'_0(p) = \begin{pmatrix}
    \bbZ_p& \bbZ_p \\ p \bbZ_p &  \bbZ_p^\times 
\end{pmatrix}.$$
In particular, they are both preserved by scalar multiplication by $\bbZ_p^\times$ and the subgroup $\Gamma_0(p \bbZ_p):= \Sigma_0(p) \cap \Sigma'_0 (p)$.

\begin{definition}
    For any $m \geq 0$, we let $A_{\calV, m}(T)$ be the space of functions
    $f : T \rightarrow \Lambda_{\calV}$
    such that 
    $$f(\gamma \cdot t) = \kappa_{\calV}(\gamma) f(t)$$
    for all $\gamma \in \bbZ_p^\times$ and $f(1 ,z) \in LA_m(\bbZ_p, \Lambda_\calV )$.
    This space is a $\Lambda_\calV$-module and we endow it with the topology induced by $m_\calV$.
    Similarly, one defines $A_{\calV, m}(T')$ to be the space of functions $f : T' \rightarrow \Lambda_\calV$ with $f(\gamma \cdot t) = \kappa_{\calV}(\gamma) f(t)$ for all $\gamma \in \bbZ_p^\times$ and $f(pz, 1) \in LA_m(\bbZ_p, \Lambda_\calV)$.
\end{definition}

\begin{definition}
    We define 
    $$D_{\calV, m}(T) := \Hom_{\Lambda_\calV}(A_{\calV, m}(T), \Lambda_\calV) $$
    and $D_{\calV, m}(T')$ in a similar manner.
\end{definition}

\begin{remark}
    As we will frequently give statements that work both for $T$ and T$'$,
    we will use the convention $T^\cdot$ as in \cite{BSV_reciprocity_balanced} to denote either $T$ or $T'$.
\end{remark}

One can also define the small versions $A_{k, m}(T^\cdot)$ and $D_{k, m}(T^\cdot)$, where the pair $(\Lambda_\calV, \kappa_\calV)$ is replaced by $(\calO_L, k)$ with $k \in \calV \cap \bbZ$.
Then one has specialization maps (cf. \cite[\S~4.1.3]{BSV_reciprocity_balanced})
$$A_{\calV, m}(T^\cdot) \xrightarrow{\rho_k} A_{k, m}(T^\cdot) \textrm{ and } D_{\calV, m}(T^\cdot) \xrightarrow{\rho_k} D_{k, m}(T^\cdot)$$
induced by quotient by the maximal ideal $\mathfrak{m}_k$ corresponding to the point $k$.

\begin{remark}[{\cite[\S~4]{BSV_reciprocity_balanced}}]
    The objects $A_{\calV, m}(T^\cdot)$ and $D_{\calV, m}(T^\cdot)$ as well as the small versions have natural actions of $\Gamma = \Gamma_1(N, p)$.
    In particular, there are notions of Hecke operators $T_\ell^\cdot, U^\cdot=  U_p^\cdot$ and Atkin--Lehner operator $w_p$ on $H^1(\Gamma, A_{\calV, m}(T^\cdot))$ and $H^1(\Gamma, D_{\calV, m}(T^\cdot))$.
\end{remark}

Later we will introduce big \'{e}tale sheaves whose construction is based on locally analytic functions and distributions.
For this purpose, we need the following lemma.

\begin{lemma}[{\cite[Lemma~4.2.7]{LZ-Rankin_Eisenstein}}]
    We may write $D_{\calV, m}(T^\cdot)$ as an inverse limit
    $$ D_{\calV, m}(T^\cdot) = \varprojlim_{n \geq 0} D_{\calV, m}(T^\cdot)/ \Fil^n D_{\calV, m}(T^\cdot)$$
    where $\Fil^n D_{\calV, m}(T^\cdot)$ is a decreasing filtration preserved by $\Sigma^\cdot_0(p)$ and each quotient is finite of $p$-power order.
\end{lemma}

\paragraph{Big \'{e}tale sheaves}

Let $Y = Y_1(N, p)$ be the affine modular curve defined over $\bbZ [1/Np]$.
Fix a geometric point $\eta: \Spec (\overline{\bbQ}) \rightarrow Y$ and denote by $\mathcal{G}$ the \'{e}tale fundamental group $\pi_1^{\mathrm{\acute e t}} (Y, \eta)$, which admits a natural map $\mathcal{G} \rightarrow \Gamma_0(p \bbZ_p)$.

Let $S_f (Y_\mathrm{\acute e t})$ be the category of locally constant constructible sheaves on $Y_{\acute e t}$ with finite stalk of $p$-power order at $\eta$.
For any topological group $G$, we denote by $M_f(G)$ the category of finite sets of $p$-power order with continuous $G$-actions.

Following \cite[\S~4.2]{BSV_reciprocity_balanced}, one sees that taking stalks at $\eta$ defines an equivalence of categories between $S_f (Y_\mathrm{\acute e t})$ and $M_f (\mathcal{G})$.
We define $M_\mathrm{cts}(G)$ to be the category of $G$-modules $M$ of the form $M = \cup_{i \in I} M_i$ with $M_i \in M_f(G)$ for all $i \in I$, and $M(G)$ to be the category of inverse systems of objects in $M_\mathrm{cts}(G)$.
We also define similarly $S_\mathrm{cts} (Y_\mathrm{\acute e t})$ and $S (Y_\mathrm{\acute e t})$.
If $G$ is either $\mathcal{G}$ or $\Gamma_0(p \bbZ_p)$, one has a functor 
$$\cdot^\mathrm{\acute e t} : M(G) \rightarrow S(Y_\mathrm{\acute e t})$$
extending the previous equivalence.
For any $\mathscr{F} = (\mathscr{F}_i)_i \in S(Y_\mathrm{\acute e t})$, we will mainly consider the following cohomology group
$$H^i_{\mathrm{\acute e t}}(Y, \mathscr{F})= \varprojlim_i H^i_{\mathrm{\acute e t}}(Y, \mathscr{F}_i).$$

We now apply the above construction to the modules, viewed as element in $M(\Gamma_0(p \bbZ_p))$ (cf. \cite[\S 4.2]{BSV_reciprocity_balanced}),
\begin{align*}
    A_{\calV, m}(T^\cdot) &= \varprojlim_n A_{\calV, m}(T^\cdot)/ m_\calV^n, \\
    D_{\calV, m}(T^\cdot)  &= \varprojlim_n D_{\calV, m}(T^\cdot)/ \Fil^n. 
\end{align*}
As a consequence, we have objects $ A_{\calV, m}(T^\cdot)^{\et}$ and $D_{\calV, m}(T^\cdot)^{\et}$ in $S(Y_{\et})$.
We will often drop the superscript $\et$ when no confusion might be made.

\begin{remark} \label{remark: overconvergent cohomology}

We conclude this subsection with a list of facts about these big \'{e}tale sheaves (cf. \cite[\S~4.2]{BSV_reciprocity_balanced}).

\begin{enumerate}
    \item One has isomorphisms
    \begin{align*}
        H^1_{\et}(Y_{\bar{\bbQ}}, A_{\calV, m}(T^\cdot)^{\et}) &\cong H^1(\Gamma,  A_{\calV, m}(T^\cdot)), \\
        H^1_{\et}(Y_{\bar{\bbQ}}, D_{\calV, m}(T^\cdot)^{\et}) &\cong H^1(\Gamma,  D_{\calV, m}(T^\cdot)), \\
        H^1_{\et, c}(Y_{\bar{\bbQ}}, D_{\calV, m}(T^\cdot)^{\et}) &\cong H^1_c(\Gamma,  D_{\calV, m}(T^\cdot)).
    \end{align*}
    where we refer the definition of $H^1_c(\Gamma,  D_{\calV, m}(T^\cdot))$ to \cite[(77)]{BSV_reciprocity_balanced}.
    \item For any continuous character $\chi : \bbZ_p^\times \rightarrow \Lambda_{\calV}^\times$, one can define the twisted modules $ D_{\calV, m}(T^\cdot)(\chi)$ and $A_{\calV, m}(T^\cdot)(\chi)$.
    Their resulting \'{e}tale sheaves will be denoted by $ D_{\calV, m}(T^\cdot)(\chi)^{\et}$ and $A_{\calV, m}(T^\cdot)(\chi)^{\et}$.
    Then one has the following identifications of $G_{\bbQ}$-modules
    \begin{align*}
        H^1_{\et}(Y_{\bar{\bbQ}}, D_{\calV, m}(T^\cdot)(\chi)^{\et}) &= H^1_{\et}(Y_{\bar{\bbQ}}, D_{\calV, m}(T^\cdot)^{\et})(\chi_\bbQ'),  \\
        H^1_{\et}(Y_{\bar{\bbQ}}, A_{\calV, m}(T^\cdot)(\chi)^{\et}) &= H^1_{\et}(Y_{\bar{\bbQ}}, A_{\calV, m}(T^\cdot)^{\et})(\chi_\bbQ')
    \end{align*}
    where $\chi_\bbQ' := \chi \circ \chi_{\mathrm{cyc}}^{-1}: G_\bbQ \rightarrow \Lambda_\calV^\times$.
    Similar statements also hold for the compactly supported versions.
    \item 
    Notice that the filtrations on $A_{\calV, m}(T^\cdot)$ and $D_{\calV, m}(T^\cdot)$ are respected by the action of $\Sigma^\cdot(p)$, so there are Hecke operators $U^\cdot$ on the $G_\bbQ$-modules 
    $$H^1(\Gamma, A_{\calV, m}(T^\cdot)), H^1(\Gamma, D_{\calV, m}(T^\cdot)), H^1_c(\Gamma, D_{\calV, m}(T^\cdot)).$$
    Moreover, $U^\cdot$ admits slope $\leq h$ decompositions for all $h \in \bbQ_{\geq 0}$ (with possibly shrinking $\calV$) on these spaces.
    \item The evaluation map $A_{\calV, m}(T^\cdot) \otimes_{\Lambda_\calV} D_{\calV, m}(T^\cdot) \rightarrow \Lambda_{\calV}$ and the trace map induce a cup-product
    $$H^1(\Gamma, A_{\calV, m}(T^\cdot)) \otimes H^1_c (\Gamma, D_{\calV, m}(T^\cdot)) \rightarrow H^2_c(\Gamma, \Lambda_\calV) \cong \Lambda_\calV,$$
    under which the action of $U^\cdot$ on $H^1(\Gamma, A_{\calV, m}(T^\cdot))$ is adjoint to the one on $H^1_c(\Gamma, D_{\calV, m}(T^\cdot))$.
    For $h \in \bbQ_{\geq 0}$ and possibly shrinking $\calV$, we have a morphism of $\scrO_\calV = \Lambda_\calV [1/p]$-modules
    $$\xi^\cdot_{\calV, m} :H^1(\Gamma, A_{\calV, m}(T^\cdot))^{\leq h} \rightarrow \Hom_{\scrO_\calV} (H^1_c(\Gamma, D_{\calV, m}(T^\cdot))^{\leq h}, \scrO_\calV).$$
    \item On the other hand, the determinant map $\det : T \times T' \rightarrow \bbZ_p^\times$ induces a cup-product pairing
    $$\det: H^1(\Gamma, D_{\calV, m}(T)) \otimes_{\Lambda_\calV} H^1_c(\Gamma, D_{\calV, m}(T')) \rightarrow \Lambda_\calV$$
    with $U$ and $U'$ adjoint to each other.
    In particular, for every $h \in \bbQ_{\geq 0}$, we have isomorphisms
    $$\zeta'_{\calV, m} : \Hom_{\scrO_\calV} (H^1_c(\Gamma, D_{\calV, m}(T')^{\leq h}, \scrO_\calV ) \cong H^1(\Gamma, H^1(\Gamma, D_{\calV, m})(T))^{\leq h}$$
    and similarly
    $$\zeta_{\calV, m} : \Hom_{\scrO_\calV} (H^1_c(\Gamma, D_{\calV, m}(T)^{\leq h}, \scrO_\calV ) \cong H^1(\Gamma, H^1(\Gamma, D_{\calV, m})(T'))^{\leq h}.$$
    \item For $h \in \bbQ_{\geq 0}$ and $\calV$ sufficiently small, the composition of $\zeta_{\calV, m}$ and $\xi_{\calV, m}$ gives a morphism of $G_\bbQ$-modules
    \begin{equation}
        s_{\calV, h} : H^1(\Gamma, A_{\calV, m}(T))^{\leq h} (\kappa_{\calV}) \rightarrow H^1(\Gamma, D_{\calV, m}(T'))^{\leq h},
    \end{equation}
    where by an abuse of notation we let $\kappa_\calV: G_\bbQ \rightarrow \Lambda_\calV^\times$ be the composition of $\chi_{\mathrm{cyc}}$ and the universal character $\kappa_\calV$.
    Moreover, for every integer $k = r+2 \in \calV \cap \bbZ$ with $h < k-1$, the specialization maps induce the following commutative diagram
    $$\begin{tikzcd}
         H^1(\Gamma, A_{\calV, m}(T))^{\leq h} (\kappa_{\calV}) \arrow[d, "\rho_k"] \arrow[r, "s_{\calV, h}"] &H^1(\Gamma, D_{\calV, m}(T'))^{\leq h} \arrow[d, "\rho_k"] \\
         H^1_{\et}(Y_{\bar{\bbQ}}, \mathscr{S}_r)_L^{\leq ' h} (r) \arrow[r,"s_r"] & H^1_{\et}(Y_{\bar{\bbQ}}, \mathscr{L}_r)_L ^{\leq ' h} 
    \end{tikzcd}$$
    where $\leq ' h$ means the slope $\leq h$ part for the operator $U'$.
    The sheaves $\mathscr{S}_r$ and $\mathscr{L}_r$ are the small \'{e}tale sheaves defined in \cite[\S~2.3]{BSV_reciprocity_balanced} and $s_r$ is the natural duality between the two cohomology groups.
    Similarly, one has an isomorphism
    \begin{equation}
        s'_{\calV, h} : H^1(\Gamma, A_{\calV, m}(T'))^{\leq h} (\kappa_{\calV}) \rightarrow H^1(\Gamma, D_{\calV, m}(T))^{\leq h}
    \end{equation}
    with an analogous commutative diagram.
    \item Suppose that $\bff$ is a Coleman family of slope $\leq a$ defined over an open disk $\calV_f$.
    For simplicity, we will write $A_{f, m}(T^\cdot)$ for $A_{\calV_f, m}(T^\cdot)$ and $D_{f, m}(T^\cdot)$ for $D_{\calV_f, m}(T^\cdot)$.
    Then, for some $m$ and possibly shrinking $\calV_f$, one defines the big Galois representation attached to $\bff$,
    $$H^1(\Gamma, D_{f, m}(T'))^{\leq h} (1) \twoheadrightarrow V(\bff) $$
    to be the maximal $\scrO_f$-quotient on which $T'_\ell$, $U'$ and $\langle d \rangle^*$ act respectively as multiplication by $a_\ell(\bff)$, $a_p(\bff)$ and $\chi_f(d)$ for all $\ell \nmid NP$ and $d \in (\bbZ/ N_f \bbZ)^\times$.
    Similarly, one defines
    $$V^*(\bff) \hookrightarrow H^1_c(\Gamma, D_{f, m}(T))^{\leq h} (-\kappa_f)$$
    to be the maximal $\scrO_f$-submodule by using the Hecke operators $T_\ell$, $U$ and $\langle d \rangle$ instead.
    
    These two $\scrO_f$-modules are free (cf. \cite[Theorem 4.6.6]{LZ-Rankin_Eisenstein}) and in fact are direct summands of $H^1(\Gamma, D_{f, m}(T'))^{\leq h} (1)$ and $H^1_c(\Gamma, D_{f, m}(T))^{\leq h} (-\kappa_f)$ respectively.
    In particular, there are natural maps
    $$\pr_f : H^1(\Gamma, D_{f, m}(T'))^{\leq h} (1) \twoheadrightarrow V(\bff) \ \textrm{ and } \pr_f^*: H^1_c(\Gamma, D_{f, m}(T))^{\leq h} (-\kappa_f) \twoheadrightarrow V(\bff)^*.$$

    If $\bff$ is a Hida family, then there is a natural short exact sequence of $\scrO_f[G_{\bbQ_p}]$-modules (cf. \cite[(99)]{BSV_reciprocity_balanced})
    $$0 \rightarrow V(\bff)^+ \rightarrow V(\bff) \rightarrow V(\bff)^- \rightarrow 0$$
    with explicit descriptions for $V(\bff)^{\pm}$.
    However, in the non-ordinary case, this exact sequence does not exist, at least not in the category of Galois representations.
    To obtain an analogous result, one needs to consider the category of $(\phig)$-modules, which is the main subject of the next subsection.
\end{enumerate}
\end{remark}

\subsection{On \texorpdfstring{$(\varphi, \Gamma)$}{}-modules and triangulations}


This subsection aims to recall several facts about $(\varphi, \Gamma)$-modules and develop necessary setups for defining the big logarithm maps for finite slope families.

A brief reference on $(\varphi, \Gamma)$-modules and explicit computations of the logarithm maps can be found in \cite{Berger_BK_exp_formula}.
For the definitions of various Selmer groups for $(\varphi, \Gamma)$-modules, we follow the notations in \cite[\S~1]{Benois_Generalization_L-invariant}.
Since we are interested in the case of modular forms, we will also recall several useful results in \cite[\S~6]{LZ-Rankin_Eisenstein}.

Let $\scrR$ be the Robba ring of $\bbQ_p$.
It is the ring consisting of formal Laurent series over $\bbQ_p$ in one variable $\pi$ which converges in some annulus of the form $\{ p^{-\epsilon} <|z|_p < 1 \} \subset \mathbb{A}^{1, \rig}$ for some $\epsilon \in \bbQ_{>0}$.
There are an action of Frobenius $\varphi$ and an action of $\Gamma := \Gal(\bbQ_p(\mu_{p^\infty})/ \bbQ_p) \cong \bbZ_p^\times$ on $\scrR$, defined by 
\begin{align*}
    \tau(\pi) &= (1+\pi)^{\chi_{\mathrm{cyc}} (\tau)} - 1, \quad \tau \in \Gamma, \\
    \varphi(\pi) &= (1+\pi)^p -1,
\end{align*}
where $\chi_{\mathrm{cyc}}$ is the cyclotomic character.
There is also a left inverse $\psi$ of $\varphi$, whose definition we would like to refer to \cite[\S~I.2]{Berger_BK_exp_formula}.
\begin{definition}
    By a $p$-adic representation of $G_{\bbQ_p}$, we always mean a finite dimensional $\bbQ_p$-vector space with continuous linear action of $G_{\bbQ_p}$.
    A $(\varphi, \Gamma)$-module $M$ over $\scrR$ is a free $\scrR$-module of finite rank with $\scrR$-semilinear actions of $\varphi$ and $\Gamma$.
    
    Given a reduced affinoid algebra $A$ over $\bbQ_p$, let $\scrR_A := \scrR \hat{\otimes}_{\bbQ_p} A$.
    By an $A$-representation of $G_{\bbQ_p}$, we mean a finitely generated locally free $A$-module $M$ with a continuous $A$-linear action of $G_{\bbQ_p}$.
    One can also define $(\varphi, \Gamma)$-modules over $\scrR_A$ similarly.
\end{definition}

As is well known, there is a functor $\mathbf{D}^\dagger_{\rig}$ from the category of $p$-adic representations of $G_{\bbQ_p}$ to the category of $(\varphi, \Gamma)$-modules over $\scrR$.
A similar functor also exists from $A$-representations to $(\varphi, \Gamma)$-modules over $\scrR_A$, which will still be denoted by $\mathbf{D}^\dagger_{\rig}$ (cf. \cite{Nakamura_2014}).


\begin{prop}[{\cite[\S~1]{Benois_Generalization_L-invariant}}]
    Let notations be as before.
    We have the following results.
    \begin{enumerate}
        \item The functor $\mathbf{D}^\dagger_{\rig}$ establishes an equivalence between the category of $p$-adic representations of $G_{\bbQ_p}$ and the category of $(\varphi, \Gamma)$-modules over $\scrR$ of slope $0$.
        \item Let $V$ be a $p$-adic representation and $D_{\bullet}$ be the functors of Fontaine for $\bullet \in \{ \dR, \cris, \mathrm{st} \}$.
        Then there are functors $\mathcal{D}_\bullet$ on the category of $(\varphi, \Gamma)$-modules over $\scrR$ such that 
        $$D_\bullet(V) = \mathcal{D}_\bullet (\mathbf{D}^\dagger_{\rig}(V))$$
        for $\bullet \in \{ \dR, \cris, \mathrm{st} \}$.
        In particular, one can define the notions of de Rham, crystalline, and semistable $(\varphi, \Gamma)$-modules over $\scrR$.
    \end{enumerate}
\end{prop}

\begin{definition}
    For any continuous character $\delta: \bbQ_p^\times \rightarrow \scrR_A^\times$,
    we let $\scrR_A \{ \delta \}$ be the module $\scrR_A \cdot e_{\delta}$, with $\varphi(e_\delta) = \delta(p) e_\delta$ and $\gamma (e_\delta) = \delta(\gamma) e_\delta$ for $\gamma \in \Gamma$.
    The object $\scrR_A \{\delta \}$ serves as an important example for $(\varphi, \Gamma)$-modules of rank $1$ over $\scrR_A$.
    If we want to separate the actions of $\varphi$ and $\Gamma$, we will use the notation  
    $\scrR_A \{\delta\} = \scrR_A (a) [b]$ where $a := \delta(p)  \in \scrR_A$ and $b := \delta |_\Gamma : \Gamma \rightarrow \scrR_A$.

\end{definition}

\begin{definition}
Fix a generator $\gamma$ of $\Gamma$.
For a $(\varphi, \Gamma)$-module $M$ over $\scrR_A$, we define the cohomology groups $H^i(\bbQ_p, M)$ to be the cohomology of the following complex
$$M \xrightarrow[]{(\varphi-1, \gamma -1)} M\oplus M \xrightarrow[]{(\gamma-1, -(\varphi -1))} M.$$
The analytic Iwasawa cohomology $H^i_{\Iw}(\bbQ_{p, \infty}, M)$ is defined as the cohomology of the complex
$$ 0 \rightarrow M \xrightarrow[]{\psi-1 } M.$$
\end{definition}

\begin{prop}[{\cite[Theorem~0.2]{Berger_representations} \& \cite[Proposition~2.3.7]{Kedlaya_cohomology_phiGamma}}]
    Let $V$ be either a $p$-adic representation or an $A$-representation. 
    Then we have 
    $$H^i(\bbQ_p, \mathbf{D}^\dagger_{\rig}(V)) = H^i(\bbQ_p, V)$$
    for all $i$, where $H^i(\bbQ_p, V)$ is the continuous cohomology.

\end{prop}

\begin{theorem}[\cite{Nakamura_2014}]
    If $A$ is a finite extension of $\bbQ_p$, and $M$ is a de Rham $(\varphi, \Gamma)$-module over $\scrR_A$.
    Then there are Bloch--Kato exponential and dual-exponential maps
    $$ \exp : \mathcal{D}_{\dR}(M) \rightarrow H^1(\bbQ_p, M), \quad \exp^* : H^1(\bbQ_p, M) \rightarrow \mathcal{D}_{\dR}(M),$$
    which are compatible with the usual definitions when $M = \Drig (V)$ for a de Rham representation $V$.
\end{theorem}

Given a $(\varphi, \Gamma)$-module $M$, similarly to the case of Galois representations, one may interpret $H^1(\bbQ_p, M)$ as isomorphism classes of extensions of $\scrR$ by $M$.
By using the functor $\mathcal{D}_{\cris}$, one can define 
$H^1_f(\bbQ_p, M)$ to be the subset of crystalline extensions (cf. \cite[\S~1.4.1]{Benois_Generalization_L-invariant}).
Its relation with the finite Bloch--Kato Selmer groups of $p$-adic representations is illustrated in the following proposition.

\begin{prop}[{\cite[Proposition~1.4.2]{Benois_Generalization_L-invariant}}]
    Suppose $V$ is a $p$-adic representation.
    Then 
    $$H^1_f(\bbQ_p, V) \cong H^1_f(\bbQ_p, \mathbf{D}^\dagger_{\rig}(V)).$$
\end{prop}
\begin{remark}
    In fact, one can also define $H^1_g$ and $H^1_e$ (cf. \cite[Definition~2.4]{Nakamura_classification}).
    They also agree with the usual Selmer groups $H^1_g$ and $H^1_e$ for $p$-adic representations.
\end{remark}


Let $\bff$ be a primitive Coleman family of finite slope $a$, defined over a wide open disk $\calV_f \subset \calW$, with nebentypus character $\chi_f$ as in \S \ref{subsection: Coleman families}.

\begin{theorem}[{\cite[Theorem~6.3.2]{LZ-Rankin_Eisenstein} \& \cite{triangulation_families}}] \label{theorem: triangulation of modular forms} \label{theorem: triangulation}
    One can find an affinoid disk $\calU_f= \mathrm{Spm}(A_f) \subset \calV_f$ such that over $\calU_f$, the $(\phig)$-module $M(\bff) := \mathbf{D}^\dagger_{\rig} (V(\bff))$ (over $\scrR_f := \scrR_{A_f}$) admits a canonical triangulation
    $$ 0 \rightarrow F^+ M(\bff) \rightarrow M(\bff) \rightarrow F^- M(\bff) \rightarrow 0,$$
    where
    \begin{enumerate}
        \item $F^- M(\bff) \cong \scrR_f(a_p(\bff))$, \textit{i.e.}, $\varphi$ acts as multiplication by $a_p(\bff)$;
        \item  $F^+ M(\bff) \cong \scrR_f (\chi_{f}(p) - a_p(\bff)) [1+ \kappa_{f}]$.
    \end{enumerate}
    Similar results hold for the dual module $M^*(\bff) := \mathbf{D}^\dagger_{\rig} (V^*(\bff))$.
    Notice that we follow the additive notation of characters as in \cite{BSV_reciprocity_balanced}.

\end{theorem}


\subsection{The balanced Selmer group}
Let $\bff, \bfg, \bfh$ be three Coleman families with tame characters satisfying the relation $\chi_f \cdot \chi_g \cdot \chi_h =1$,
and let $a_f, a_g, a_h$ be their respective slopes.
Suppose $\calV_f, \calV_g, \calV_h$ are defined over some finite extension $L$ of $\bbQ_p$.
We fix centers $k_0, \ell_0, m_0$ and uniformizers $\mathbf{k}-k_0, \bm{\ell}- \ell_0, \mathbf{m}-m_0$ for $\calV_f, \calV_g, \calV_h$ respectively (cf. Definition \ref{definition: open disk}).
We set $\scrO_{fgh} := \scrO_f \hat{\otimes}_L \scrO_g \hat{\otimes}_L \scrO_h$ and $\scrR_{fgh} := \scrR_f \hat{\otimes}_L \scrR_g \hat{\otimes}_L \scrR_h$.

Consider the character $\Xi: G_{\bbQ_p} \rightarrow \scrO_{fgh}^\times$ defined by
\begin{align*}
    \Xi(g) :&= ( \omega^{(4 - k_0 -l_0 -m_0)/2} \cdot \langle \phantom{e} \rangle^{(4 - \mathbf{k} - \bm{\ell} - \mathbf{m})/2}) \circ \chi_{\mathrm{cyc}}(g) \\
    &= \frac{1}{2} (- \kappa_f- \kappa_g- \kappa_h- 2) \circ \chi_{\mathrm{cyc}}(g).
\end{align*}
We set 
$$V(\bff, \bfg, \bfh):=  V(\bff) \hat{\otimes}_L V(\bfg) \hat{\otimes}_L V(\bfh) \otimes_{\scrO_{fgh}} \Xi,$$
which is Kummer self-dual, and set $M(\bff, \bfg, \bfh):= \Drig(V(\bff, \bfg, \bfh))$.
One defines a filtration on $M(\bff)$ by $F^0 M(\bff) = M(\bff)$, $F^1 M(\bff) = F^+M(\bff)$, $F^2M(\bff) =0$, and similarly for $M(\bfg), M(\bfh)$. 
These filtrations induce a filtration on $M(\bff, \bfg, \bfh)$, defined by 
$$F^n M(\bff, \bfg, \bfh) := \left [ \sum_{p +q+r =n} F^p M(\bff) \hat{\otimes} F^q M(\bfg) \hat{\otimes} F^r M(\bfh) \right ] \{\Xi \}$$
where by $\{ \Xi \}$ we mean the image of $\Xi$ under the functor $\Drig$ (cf. \cite[\S~1.5]{Benois_Generalization_L-invariant}).

\begin{definition}
    Analogous to \cite[\S~7.2]{BSV_reciprocity_balanced}, we define the balanced Selmer group to be 
    \begin{equation}
        H^1_{\mathrm{bal}}(\bbQ_p, M(\bff, \bfg, \bfh) ):= H^1(\bbQ_p, F^2 M(\bff, \bfg, \bfh)) \subset H^1(\bbQ_p, M(\bff, \bfg, \bfh)).
    \end{equation}
\end{definition}
    We remark that by examining the $(\varphi, \Gamma)$-module structure of $\textrm{gr}^0 M(\bff, \bfg, \bfh)$, the natural map $F^2 M(\bff, \bfg, \bfh) \hookrightarrow M(\bff, \bfg, \bfh)$ induces an injection on $H^1$ (cf. \cite[(148)]{BSV_reciprocity_balanced}).

One important property of this balanced Selmer group is the following lemma.
\begin{lemma}[{cf. \cite[Lamma~7.2]{BSV_reciprocity_balanced}}]
    If $(x, y, z) \in \Sigma_{\mathrm{bal}}$ is a balanced triple of classical weights, then
    \begin{equation}
        H^1_{\mathrm{bal}}(\bbQ_p, M(\bff_x, \bfg_y, \bfh_z) ) = H^1_f(\bbQ_p, M(\bff_x, \bfg_y, \bfh_z)) = H^1_f(\bbQ_p, V(\bff_x, \bfg_y, \bfh_z)) 
    \end{equation}
    where the last one is the Bloch--Kato finite Selmer group.
    As a consequence, one has the Bloch--Kato exponential map
    $$\exp_p : D_{\dR}(V(\bff_x, \bfg_y, \bfh_z)) / \Fil^0 \cong  H^1_{\mathrm{bal}}(\bbQ_p, M(\bff_x, \bfg_y, \bfh_z) )$$
    and we denote its inverse by $\log_p$.
\end{lemma}
\begin{proof}
    The proof is essentially identical to that of \cite[Lamma~7.2]{BSV_reciprocity_balanced}, which boils down to compute the Hodge--Tate weights, or the actions of $\Gamma$ in the case of $(\varphi, \Gamma)$-modules.
    In particular, one has 
    $$ \Fil^0 \mathcal{D}_{\dR} (F^2 M(\bff_x, \bfg_y, \bfh_z)) =0 \textrm{ and } \Fil^0 \mathcal{D}_{\dR} (M(\bff_x, \bfg_y, \bfh_z) /F^2 ) = \mathcal{D}_{\dR}(M(\bff_x, \bfg_y, \bfh_z)).$$
    The second identity implies that $H^1_e(\bbQ_p, M(\bff_x, \bfg_y, \bfh_z) /F^2 ) = 0$.
    Notice that $F^2 M(\bff_x, \bfg_y, \bfh_z)$ is Kummer dual to $M(\bff_x, \bfg_y, \bfh_z) /F^2$,
    so one has 
    $$H^1(\bbQ_p, F^2 M(\bff_x, \bfg_y, \bfh_z)) = H^1_g(\bbQ_p, F^2 M(\bff_x, \bfg_y, \bfh_z)).$$
    By \cite[Lemma~3.5]{BSV_reciprocity_balanced}, $H^1_g(\bbQ_p, F^2 M(\bff_x, \bfg_y, \bfh_z)) = H^1_f(\bbQ_p, F^2 M(\bff_x, \bfg_y, \bfh_z))$.
    One then deduces that $H^1_{\bal}(\bbQ_p, F^2 M(\bff_x, \bfg_y, \bfh_z)) \subset H^1_f(\bbQ_p, F^2 M(\bff_x, \bfg_y, \bfh_z))$.
    For the other direction of the inclusion, it follows similarly by using the fact that $\Fil^0 \mathcal{D}_{\dR} (F^2 M(\bff_x, \bfg_y, \bfh_z)) =0$.
\end{proof}

If we set \[
    M(\bff, \bfg, \bfh)_f := F^-M( \bff) \hat{\otimes} F^+ M(\bfg) \hat{\otimes} F^+M(\bfh)\{ \Xi \},
\]
and similarly for $M(\bff, \bfg, \bfh)_g$ and $M(\bff, \bfg, \bfh)_h$, then
the graded piece $F^2 / F^3$ decomposes as 
\begin{equation}
    F^2/ F^3 =  M(\bff, \bfg, \bfh)_f \oplus M(\bff, \bfg, \bfh)_g \oplus M(\bff, \bfg, \bfh)_h.
\end{equation}
As a result, there are natural projections
$F^2 M(\bff, \bfg, \bfh) \rightarrow  M(\bff, \bfg, \bfh)_\xi$ 
for $\xi \in \{f, g, h\}$.

\subsection{The three-variable big logarithm}
Write $V_{\dR}(\bff_x, \bfg_y, \bfh_z) := D_{\dR} (V(\bff_x, \bfg_y, \bfh_z))$ and similarly for its Kummer dual $V_{\dR}^*(\bff_x, \bfg_y, \bfh_z)$.

If $w =(x, y, z) \in \Sigma_{\mathrm{bal}}$, one can find an element $\eta^a_{\bff_x} \otimes \omega_{\bfg_y} \otimes \omega_{\bfh_z} \in \Fil^0 V_{\dR}^*(\bff_x, \bfg_y, \bfh_z)$ similarly to the construction in \S \ref{subsection: AJ imgae} (cf. \cite[\S~3.1]{BSV_reciprocity_balanced}).
The Bloch--Kato logarithm 
$$\log_p: H^1_{\mathrm{bal}}(\bbQ_p, M(\bff_x, \bfg_y, \bfh_z)) \xrightarrow{\sim} D_{\dR}(V(\bff_x, \bfg_y, \bfh_z)) \cong [\Fil^0 V_{\dR}^*(\bff_x, \bfg_y, \bfh_z)]^\vee $$
then enables us to define $\log_{p, f}: H^1_{\mathrm{bal}}(\bbQ_p, M(\bff_x, \bfg_y, \bfh_z)) \rightarrow L$ as the composition of $\log_p$ and evaluation at $\eta^a_{\bff_x} \otimes \omega_{\bfg_y} \otimes \omega_{\bfh_z}$.

On the other hand, if $w \in \Sigma_f$, one has the Bloch--Kato dual exponential 
$$ \exp_p^* : H^1(\bbQ_p, M(\bff_x, \bfg_y, \bfh_z)) \rightarrow \Fil^0 V_{\dR}(\bff_x, \bfg_y, \bfh_z).$$
We then let $\exp_{p, f}^*$ be the composition of $\exp_p^*$ with evaluation at $\eta^a_{\bff_x} \otimes \omega_{\bfg_y} \otimes \omega_{\bfh_z}$.
Notice that here one needs to identify $\Fil^0 V_{\dR}(\bff_x, \bfg_y, \bfh_z)$ as a subspace of $[V^*_{\dR} (\bff_x, \bfg_y, \bfh_z)]^\vee$ via Poincar\'{e} duality.

Before stating the proposition about the three-variable big logarithm, we would like to recall some notations.
Let $w = (x, y, z) \in \Sigma_{\cl}$, we write $\alpha_{f_x} = a_p(\bff)_x$ and $\beta_{f_x} = \chi_f(p) p^{x-1} / \alpha_{f_x}$ and similarly for $g$ and $h$.
We also let $c_w := (x +y +z -2)/2$ be the point of symmetry.

\begin{prop}\label{prop: three variable log}
    There is a unique $\scrO_{fgh}$-module morphism
    $$ \scrL_f = \scrL og_f (\bff, \bfg, \bfh) : H^1_{\bal}(\bbQ_p, M(\bff, \bfg, \bfh )) \rightarrow \scrO_{fgh} $$
    such that, for all $w = (x, y, z) \in \Sigma_{\cl}$ with $\alpha_{f_x} \beta_{g_y} \beta_{h_z} \neq p^{c_w}$ and $Z \in H^1_{\mathrm{bal}}(\bbQ_p, M(\bff, \bfg, \bfh ))$,
    \begin{equation} \label{equation: 3 variable big log}
        \scrL_f(Z)_w = (p-1) \alpha_{f_x} \cdot 
        \frac{\left ( 1 - \frac{ \beta_{f_x} \alpha_{g_y} \alpha_{h_z} }{p^{c_w}}\right ) }{\left ( 1 - \frac{ \alpha_{f_x} \beta_{g_y} \beta_{h_z} }{p^{c_w}}\right ) } \cdot 
        \begin{cases}
            \frac{(-1)^{c_w-x}}{(c_w -x)!} \log_{p, f}(Z_w) &\textrm{ if } w \in \Sigma_{\bal} \\
            ( x- c_w -1)! \exp_{p, f}^* (Z_w) &\textrm{ if } w \in \Sigma_{f}
        \end{cases},
    \end{equation}
    where the subscript $w$ means the specialization at $w$.
    Moreover, $\scrL_f$ factors through the natural map
    $$H^1_{\bal}(\bbQ_p, M(\bff, \bfg, \bfh)) \rightarrow H^1(\bbQ_p, M(\bff, \bfg, \bfh)_f).$$
    Analogous results also hold for $g$ and $h$.
\end{prop}

\begin{proof}
    This result is a generalization of \cite[Proposition 7.3]{BSV_reciprocity_balanced}.
    As we are dealing with Coleman families, we will follow the approach which is close to \cite[\S 6, 7]{LZ-Rankin_Eisenstein}.
    
    An argument similar to the proof of \cite[Theorem 7.4.1]{LZ-Rankin_Eisenstein} provides a four-variable big logarithm map (with the source being $H^1_{\Iw}(\bbQ_{p, \infty}, M^\circ (\bff, \bfg, \bfh)_f )$) analogous to \cite[Proposition 7.1]{BSV_reciprocity_balanced}, where
    \[ M^\circ (\bff, \bfg, \bfh)_f := F^-M( \bff) \hat{\otimes} F^+ M(\bfg) \hat{\otimes} F^+M(\bfh). \]
    
    To get a three-variable big logarithm, we restrict the four-variable one along the character $\Xi$ (\textit{i.e.}, restricts to the hyperplane defined by $2 \mathbf{j} = \mathbf{k} -\bm{\ell} - \mathbf{m}$ as in the proof of \cite[Proposition 7.3]{BSV_reciprocity_balanced}).
    By \cite[Proposition 4.2.3]{Kedlaya_cohomology_phiGamma}, we have the isomorphism (cf. \cite[equation (152)]{BSV_reciprocity_balanced}
    \[
        H^1_{\Iw}(\bbQ_{p, \infty}, M^\circ (\bff, \bfg, \bfh)_f ) \otimes_\Xi \scrO_{fgh} \cong H^1(\bbQ_p, M(\bff, \bfg, \bfh)).
    \]
    The rest of the proof then follows.
\end{proof}

\begin{remark}

    We end this subsection by explaining the factor
    $$(p-1) \alpha_{f_x} \cdot {\left ( 1 - \frac{ \beta_{f_x} \alpha_{g_y} \alpha_{h_z} }{p^{c_w}}\right ) }{\left ( 1 - \frac{ \alpha_{f_x} \beta_{g_y} \beta_{h_z} }{p^{c_w}}\right ) }^{-1}$$
    in (\ref{equation: 3 variable big log}).
    It suffices to understand the structure of the rank-$1$ $(\varphi, \Gamma)$-module $M(\bff, \bfg, \bfh)_f$.
    To be more precise, one needs to compute the $\varphi$-eigenvalue of $M(\bff, \bfg, \bfh)_f$ when specialized to $w =(x, y, z) \in \Sigma_{\bal}$.
    According to \cite[\S~7.1]{BSV_reciprocity_balanced}, one expects this eigenvalue to be
    $$\frac{\chi_g(p) \chi_h(p) \alpha_{f_x}}{\alpha_{g_y} \alpha_{h_z} } = \frac{p^{k-1} }{\beta_{f_x}\alpha_{g_y} \alpha_{h_z} } $$
    up to a twist of $p^j$ for some $j= j(w) \in \bbZ$.

    First, it follows from the description in Theorem \ref{theorem: triangulation of modular forms} that the term
    $$\frac{\chi_g(p) \chi_h(p) \alpha_{f_x}}{\alpha_{g_y} \alpha_{h_z} }$$
    appears.
    One now calculates the Hodge--Tate weight of $(M(\bff, \bfg, \bfh)_f )_w$, which is $j := \frac{x-y-z}{2}$.
    Then by \cite[Proposition~1.5.8]{Benois_Generalization_L-invariant}, the $\varphi$-eigenvalue on $(M(\bff, \bfg, \bfh)_f )_w$ is 
    \begin{equation}
        A:= \frac{\chi_g(p) \chi_h(p) \alpha_{f_x}}{\alpha_{g_y} \alpha_{h_z} } \cdot p^{-j} = \frac{p^{k-1} }{\beta_{f_x}\alpha_{g_y} \alpha_{h_z} } \cdot p^{-j}.
    \end{equation}
    Notice that $k-1-j = c_w$.
    Hence, one gets the Euler factor
    \begin{equation*}
        (1 -A^{-1}) (1 -p^{-1}A)^{-1} = {\left ( 1 - \frac{ \beta_{f_x} \alpha_{g_y} \alpha_{h_z} }{p^{c_w}}\right ) }{\left ( 1 - \frac{ \alpha_{f_x} \beta_{g_y} \beta_{h_z} }{p^{c_w}}\right ) }^{-1}.
    \end{equation*}
    On the other hand, it follows from Eichler--Shimura relation in families (cf. \cite[Theorem~6.4.1 \& Corollary~6.4.3]{LZ-Rankin_Eisenstein}) that one can also $p$-adically interpolate the class $\eta^a_{\bff_x}, \omega_{\bfg_y},  \omega_{\bfh_z}$.
    Here we follow the normalization in \cite[(132)]{BSV_reciprocity_balanced}.
    Namely, the big class $\eta^a_{\bff}$ is such that for $x \in \Sigma_f$, 
    $$(\eta^a_{\bff})_x = (p-1) \alpha_{f_x} \eta^a_{\bff_x}.$$
    For $\omega_{\bfg}$ and $\omega_{\bfh}$, they interpolate $\omega_{\bfg_y}$ and $ \omega_{\bfh_z}$ respectively, without extra multiples.
    This explains the factor $(p-1) \alpha_{f_x}$.

\end{remark}

\subsection{The diagonal class \texorpdfstring{$\underline{\kappa}$}{}}

We now recall the definition of the big diagonal class $\underline{\kappa}$, which interpolates (the $p$-stabilization of) the classes given by $\Delta_{x, y, z}$'s in the \'{e}tale cohomology groups.
If one examines the construction in \cite[\S~8.1]{BSV_reciprocity_balanced} carefully, one sees that their argument is also valid for the finite slope case, with only few modifications needed.

Let $\bff, \bfg, \bfh$ be as before.
We fix an integer $m$ large enough and set $A_\xi ^\cdot = A_{\calV_\xi, m} (T^\cdot)$ and $D_\xi^\cdot = D_{\calV_\xi, m} (T^\cdot)$ for $\xi \in \{ f, g, h\}$.

Set 
$(T \times T)_0:= \{ (t_1, t_2) \in T \times T \mid \det(t_1, t_2) \in \bbZ_p^\times \}$
and let $(T \times T)^0$ be the complement of $(T \times T)_0$ in $T \times T$.
Both $(T \times T)_0$ and $(T \times T)^0$ are compact open in $T \times T$ and are preserved by $\Gamma_0(p \bbZ_p)$.
The decomposition $T \times T = (T \times T)_0 \sqcup (T \times T)^0$ induces a decomposition of $\Gamma_0(p \bbZ_p)$-modules
$$ A_g \hat{\otimes}_{\calO_L} A_h = (A_g \hat{\otimes}_{\calO_L} A_h)_0 \oplus (A_g \hat{\otimes}_{\calO_L} A_h)^0$$
where $(A_g \hat{\otimes}_{\calO_L} A_h)_0$ consists of locally analytic functions supported on $(T \times T)_0$ and similarly for $(A_g \hat{\otimes}_{\calO_L} A_h)^0$.
Let $\Lambda_{fgh} := \Lambda_f \hat{\otimes}_{\calO_L} \Lambda_g \hat{\otimes}_{\calO_L} \Lambda_h$ and define the following characters from $\bbZ_p^\times$ to $\Lambda_{fgh}^\times$
\begin{align*}
    \kappa_f^* (u) &:= \omega(u)^{(\ell_0 + m_0 - k_0 -2)/ 2 } \cdot \langle u \rangle ^{(\bm{\ell} + \bm{m} - \bm{k} -2)/ 2 },\\
    \kappa_{fgh}^* (u) &:= \omega(u)^{( k_0 + \ell_0 + m_0 -6)/ 2 } \cdot \langle u \rangle ^{( \bm{k} + \bm{\ell} + \bm{m}  -6)/ 2 }.
\end{align*}
One also defines $\kappa_g^*, \kappa_h^*$ similarly so that $\kappa_{fgh}^* = \kappa_f^* + \kappa_g^* + \kappa_h^*$.

Consider now the function
$$\bm{Det} : T' \times T \times T \rightarrow \Lambda_{fgh}^\times$$
which is identically zero on $T' \times (T \times T)^0$ and on an element $(t_1, t_2 ,t_3) \in T' \times (T \times T)_0$ takes the value
$$\bm{Det}(t_1, t_2, t_3) = \det(t_1, t_2)^{\kappa_h^*} \cdot \det(t_1, t_3)^{\kappa_g^*} \cdot \det(t_2, t_3)^{\kappa_f^*}.$$
By definition, one sees that $\bm{Det} \in A'_{\calV_f, m} \hat{\otimes}_{\calO_L} A_{\calV_g, m} \hat{\otimes}_{\calO_L} A_{\calV_h, m} ( -\kappa_{fgh}^*)$ for $m$ large enough.
Hence, $\bm{Det}$ belongs to $A'_f \hat{\otimes}_{\calO_L} A_g \hat{\otimes}_{\calO_L} A_h ( -\kappa_{fgh}^*)$ in our notations.

By Remark \ref{remark: overconvergent cohomology}, it corresponds to a class
\begin{equation}
    \bm{Det}^{\et} \in H^0_{\et}(Y, A'_f \otimes A_g \otimes A_h (-\kappa_{fgh}^*))
\end{equation}
where we use the same notation $A^\cdot_\xi$ for its corresponding \'{e}tale sheaf.

Let $d: Y \rightarrow Y^3$ be the diagonal embedding.
Define $\underline{\kappa}^{\circ \circ} := \bm{\mathrm{AJ}}_{\et} (\bm{Det}^{\et} )$, where $\bm{\mathrm{AJ}}_{\et}$ denotes the following composition of maps:
\begin{align}
\begin{split}
    H^0_{\et}(Y, A'_f \otimes A_g \otimes & A_h (-\kappa_{fgh}^*)) \xrightarrow{d_*}  H^4_{\et}(Y^3, A'_f \boxtimes A_g \boxtimes A_h (-\kappa_{fgh}^*) \otimes_{\bbZ_p} \bbZ_p(2)) \\
    \xrightarrow{\mathrm{HS}} & H^1( \bbQ, H^3_{\et} (Y^3_{\bar{\bbQ}}, A'_f \boxtimes A_g \boxtimes A_h)(2 +\kappa_{fgh}^*))\\
    \xrightarrow{K} & H^1(\bbQ, H^1(\Gamma, A'_f) \hat{\otimes}_L H^1 (\Gamma, A_g) \hat{\otimes}_L H^1 (\Gamma, A_h)(2 +\kappa_{fgh}^*)) \\
    \xrightarrow{w_p \otimes \id \otimes \id} & H^1(\bbQ, H^1(\Gamma, A_f) \hat{\otimes}_L H^1 (\Gamma, A_g) \hat{\otimes}_L H^1 (\Gamma, A_h)(2 +\kappa_{fgh}^*)) \\
    \xrightarrow{s_{fgh}} & H^1(\bbQ, H^1(\Gamma, D'_f)^{\leq{a_f}} \hat{\otimes}_L H^1 (\Gamma, D'_g)^{\leq{a_g}} \hat{\otimes}_L H^1 (\Gamma, D'_h)^{\leq{a_h}}(2 -\kappa_{fgh}^*)) \\
    \xrightarrow{\pr_{fgh}} & H^1 (\bbQ, V(\bff) \hat{\otimes}_L V(\bfg) \hat{\otimes}_L V(\bfh) (-1- \kappa_{fgh}^*) ) = H^1(\bbQ, V(\bff, \bfg, \bfh)).
\end{split}
\end{align}
Here $\kappa_{fgh}^*: G_\bbQ \rightarrow \Lambda_{fgh}^\times$ denotes, by an abuse of notation, the composition of the $p$-adic cyclotomic character with $\kappa_{fgh}^*$. 
The map $d_*$ is the push-forward map.
The map $\mathrm{HS}$ is the map arising from the Hochschild--Serre spectral sequence (since $H^4_{\et}(Y^3_{\bar{\bbQ}}, \cdot \ )$ vanishes) and $2.$ of Remark \ref{remark: overconvergent cohomology}.
The map $K$ comes from the K\"{u}nneth decomposition and the isomorphism in $1.$ of Remark \ref{remark: overconvergent cohomology}.
The meaning of the map $w_p \otimes \id \otimes \id$ is clear, with $w_p: H^1(\Gamma, A'_f) \rightarrow H^1(\Gamma, A_f)$ being the Atkin--Lehner involution.
The map $s_{fgh}$ is induced by the tensor product of the $G_{\bbQ}$-morphisms
$$ H^1(\Gamma, A_\xi) \twoheadrightarrow H^1(\Gamma, A_\xi)^{\leq a_\xi} \xrightarrow{s_\xi} H^1(\Gamma, D'_\xi)^{\leq a_\xi} ( -\kappa_\xi)$$
for $\xi \in \{ f, g, h \}$.
The last map $\pr_{fgh}$ is the tensor product of the maps $\pr_{\xi}$ in $7$. of Remark \ref{remark: overconvergent cohomology}.

We now define 
$$\underline{\kappa}^\circ (\bff, \bfg, \bfh) := \frac{1}{a_p(\bff)} \cdot \underline{\kappa}^{\circ \circ} \in H^1(\bbQ, V(\bff, \bfg, \bfh))$$
and let $\underline{\kappa}^{\circ}_p := res_p (\underline{\kappa}^\circ) \in H^1(\bbQ_p, V(\bff, \bfg, \bfh))$ be its restriction to $G_{\bbQ_p}$.
Finally, we let 
$$\underline{\kappa}_p \in H^1(\bbQ_p, M(\bff, \bfg, \bfh))$$ 
be the corresponding element of $\underline{\kappa}^{\circ}_p$.

\begin{theorem}[cf. {\cite[Theorem~A]{BSV_reciprocity_balanced}}] \label{theorem: diagonal class}
    The class $\underline{\kappa}_p$ belongs to $H^1_{\bal}(\bbQ_p, M(\bff, \bfg, \bfh))$ and 
    $$\scrL og_f (\underline{\kappa}_p) = \scrL^f_p(\bff, \bfg, \bfh).$$
\end{theorem}

\begin{proof}
    The proof is the same as in \cite[\S~8]{BSV_reciprocity_balanced}.
    Through \S 8.2 to \S 8.4 in \textit{loc. cit.}, one attempts to describe the specialization of $\scrL og_f (\underline{\kappa}_p)$ at a balanced triple $w =(x, y, z) \in \Sigma_{\bal}$ in terms of $\log_p (\kappa(f_x^0, g_y^0, h_z^0))$,
    where $\kappa(f_x^0, g_y^0, h_z^0)$ is the small diagonal class defined in \cite[\S~3]{BSV_reciprocity_balanced}.
    In the end, one gets an equality
    $$\scrL og_f (\underline{\kappa}_p)_w = (-1)^{t-1} \frac{\scrE(f_x^0, g_y^0, h_z^0) }{(t-1)!\scrE_0(f_x^0)\scrE_1(f_x^0)} \cdot \log_p(\kappa(f_x^0, g_y^0, h_z^0)) (\eta^a_{f_x^0} \otimes \omega_{g_y^0} \otimes \omega_{h_z^0}).$$
    
    On the other hand, the Bloch--Kato logarithm $\log_p$ can be interpreted as the $p$-adic Abel--Jacobi map (cf. \cite[\S~1.6.1]{besser_Heidelberg} and \cite[Proposition~9.11]{synreg1}).
    In particular, 
    $$\log_p(\kappa(f_x^0, g_y^0, h_z^0)) (\eta^a_{f_x^0} \otimes \omega_{g_y^0} \otimes \omega_{h_z^0}) = \AJ_p(\Delta_{x, y, z}) (\eta^a_{f_x^0} \otimes \omega_{g_y^0} \otimes \omega_{h_z^0}).$$
    Now, Theorem \ref{theorem: main theoerem} implies that 
    $$\scrL og_f (\underline{\kappa}_p)_w = \scrL^f_p (\bff, \bfg, \bfh)(w) $$
    for all $ w \in \Sigma_{\bal}$ with the condition $\alpha_{f_x} \beta_{g_y} \beta_{h_z} \neq p^{c_w}$.
    As such a subset of weights is dense in the weight space $\Sigma$, one obtains the desired identity. 
\end{proof}

We now shift our attention to the unbalanced region $\Sigma_f$.
For simplicity, we will only deal with $w = (x, y, z) \in \Sigma_f$ that is not exceptional, \textit{i.e.}, the $p$-adic $L$-functions does not have trivial zeros (cf. \cite[\S~1.2]{BSV_reciprocity_balanced}).
Following the same argument as in \cite[\S~9.1]{BSV_reciprocity_balanced}, we have the following corollary.
\begin{cor}[{\cite[Theorem~B]{BSV_reciprocity_balanced}}] \label{cor: crystalline at unbalanced}
    Let $w =(x, y, z) \in \Sigma_f$ be a unbalanced triple that is not exceptional.
    Then the class $\underline{\kappa}_p(x, y, z)$ is crystalline (\textit{i.e.}, belongs to $H^1_f$) if and only if the complex $L$-function $L(f_x^0, g_y^0, h_z^0 ;s)$ vanishes at $s =\frac{x+y+z-2}{2}$.
\end{cor}


With Corollary \ref{cor: crystalline at unbalanced}, one can also establish an analogue of Theorem \ref{theorem: Selmer group}.
The proof will be identical to the one given in \cite[\S~6.4]{DR2} and we leave it to the reader.

\bibliographystyle{alpha}
\bibliography{ref.bib}

\end{document}